\documentclass[reqno,11pt]{amsart}
\usepackage{amssymb}
\usepackage[dvipsnames]{xcolor}
\usepackage[all,cmtip]{xy}
\usepackage{amsmath}
\usepackage{amsthm}
\usepackage{amscd}
\usepackage{mathtools, nccmath}
\usepackage{amsfonts, esint}
\usepackage{tikz}
\usetikzlibrary{matrix,arrows,arrows.meta,decorations.pathmorphing}
\usepackage{bbm,bm}
\usepackage{mathrsfs}
\usepackage[unicode=true,pagebackref=true,breaklinks=true,colorlinks=true]{hyperref}
\hypersetup{pdfauthor={Name}}
\usepackage{mathrsfs}
\DeclareMathAlphabet{\mathscrbf}{OMS}{mdugm}{b}{n}
\DeclareMathAlphabet{\mathbfcal}{OMS}{cmsy}{b}{n}

\usepackage[margin=1in]{geometry}
\usepackage{accents}

\usepackage{upgreek}
\newtheorem{dummy}{anything}[section]
\newtheorem{theorem}[dummy]{Theorem}

\newtheorem{corollary}[dummy]{Corollary}

\theoremstyle{definition}
\newtheorem{definition}[dummy]{Definition}
\newtheorem{example}[dummy]{Example}

\newtheorem{remark}[dummy]{Remark}
\newtheorem{notation}[dummy]{Notation}

\newcommand{\smalltimes}[1][0.75]{%
	\mathbin{\vcenter{\hbox{\scalebox{#1}{$\times$}}}}}

\newcommand{\colim}{\operatornamewithlimits{colim}}

\newcommand{\lto}{\longrightarrow}

\newcommand{\cA}{\mathcal A}
\newcommand{\cB}{\mathcal B}
\newcommand{\cC}{\mathcal C}
\newcommand{\cD}{\mathcal D}

\newcommand{\cG}{\mathcal G}

\newcommand{\ii}{\mathcal I}

\newcommand{\cM}{\mathcal M}
\newcommand{\cT}{\mathscr{T}}

\newcommand{\TT}{\mathbb{T}}
\newcommand{\VV}{{\mathscrbf{V}}}

\newcommand{\cS}{\mathcal S}

\newcommand{\bbN}{\mathbb N}

\newcommand{\bbR}{\mathbb R}
\newcommand{\bbS}{\mathbb S}
\newcommand{\bbZ}{\mathbb Z}
\newcommand{\vunit}{\mathbbm{1}}
\newcommand{\unit}{\mathfrak{{1}}}
\newcommand{\cunit}{\mathbf{1}}
\newcommand{\bfV}{\mathbf H}

\DeclareMathOperator{\VCat}{{\mathscrbf{V}}\mathbf{Cat}}
\newcommand{\cAb}{\mathcal Ab}

\DeclareMathOperator{\Tr}{T}
\DeclareMathOperator{\ho}{Ho}
\DeclareMathOperator{\co}{co}

\DeclareMathOperator{\Stab}{Stab}
\DeclareMathOperator{\coStab}{coStab}

\DeclareMathOperator{\Fun}{Fun}

\DeclareMathOperator{\Inv}{Inv}
\DeclareMathOperator{\coInv}{coInv}

\DeclareMathOperator{\RelCat}{\mathbf{RelCat}}
\DeclareMathOperator{\Mon}{\mathbf{Mon}}
\DeclareMathOperator{\Sym}{\mathbf{Sym}}
\DeclareMathOperator{\Cat}{\mathbf{Cat}}
\DeclareMathOperator{\CAT}{\mathbf{CAT}}

\DeclareMathOperator{\MonCat}{\mathbf{MonCat}}

\DeclareMathOperator{\Top}{\mathbf{Top}}

\DeclareMathOperator{\Psd}{Psd}

\DeclareMathOperator{\stVM}{st\ii\frM}
\DeclareMathOperator{\hml}{hml}
\DeclareMathOperator{\cohml}{cohml}

\DeclareMathOperator{\HML}{HML}
\DeclareMathOperator{\COHML}{COHML}

\DeclareMathOperator{\op}{op}
\DeclareMathOperator{\Lax}{Lax}

\DeclareMathOperator{\Adj}{Adj}

\newcommand{\HH}{\mathbf{B}({\ii{\smalltimes}\ii^{\op}})} 

\newcommand{\frM}{\bm{\mathfrak{M}}}
\newcommand{\frT}{\bm{\mathfrak{T}}}
\newcommand{\SP}{\mathcal{S}p}
\newcommand{\Set}{\mathcal{S}et}

\newcommand{\id}{\mathrm{id}}

\begin{document}

\title[Actions of Monoidal Categories]{Stabilization and costabilization with respect to an action of a monoidal category}
	\thanks{Partially supported by T\"UB\.ITAK-TBAG$/117$F$085$}
	
		\author{MEHMET AK\. IF ERDAL and \" OZG\" UN \" UNL\" U }

\address[1]{%
	Deptartment of Mathematics\\
	Yeditepe University\\
 34755,	İstanbul\\
	Turkey}
\email{mehmet.erdal@yeditepe.edu.tr}	
	\address[2]{%
		Deptartment of Mathematics\\
		Bilkent University\\
	 06800,	Ankara\\
		Turkey}
\email{unluo@fen.bilkent.edu.tr}
	
	
	
	\date{}

\begin{abstract} We study actions of monoidal categories on objects in a suitably enriched $2$-category, and applications in stable homotopy theory. Given a   monoidal category $\mathcal{I}$ and  an $\mathcal{I}$-object $\mathcal{A}$, the (co)stabilization of   $\mathcal{A}$ is obtained by universally forcing the $\mathcal{I}$-action to be reversible so that every object of $\mathcal{I}$ acts on $\mathcal{A}$ by auto-equivalences. We introduce a notion of $\ii$-equivariance for morphisms between $\mathcal{I}$-objects and give constructions of stabilization and costabilization in terms of weak ends and coends in an enriched $2$-category of $\mathcal{I}$-objects and $\mathcal{I}$-equivariant morphisms. We observe that the stabilization of a relative category with respect to an action coincides with the usual notion of stabilization in stable homotopy theory when the action is defined by loop space functors. We show that several examples that exist in the literature, including various categories of spectra, fit into our setting after fixing $\mathcal{A}$ and the $\mathcal{I}$-action on it. In particular, categories of sequential spectra, coordinate free spectra, genuine equivariant spectra, parameterized spectra indexed by vector bundles are obtained in terms of weak ends in the $2$-category of relative categories. On the other hand, the costabilization of a relative category with respect to an action gives a stable relative category akin to a version Spanier-Whitehead category. In particular, we establish a form of duality between constructions of stable homotopy categories by spectra and by Spanier-Whitehead like categories.\end{abstract}
\maketitle


\section{Introduction}
\label{sect:Introduction}  
A  pointed homotopy theory $\cA$ (e.g., pointed model category, pointed relative category,  pointed $(\infty,1)$-category, pointed derivator, etc.) equipped with a family $\{\Omega^\alpha:\cA\to \cA\}_\alpha$ of loop space functors  is called  stable if it is stable under these functors; i.e., $\Omega^\alpha$ is an auto-equivalence for each $\alpha$. If this is not the case, the stabilization of $\cA$ is obtained by formally inverting all these loop space functors. The standard construction is given by the category generated by spectrum objects in $\cA$. If there is only one such loop space functor to invert, a spectrum object consists of a sequence of objects $X_n$ in $\cA$ together with a sequence of isomorphisms $ X_n\to \Omega X_{n+1}$. The category of such objects, the stabilization of $\cA$, is equivalent the $(\infty,1)$-limit of the sequence $\cdots\stackrel{\Omega}\to \cA\stackrel{\Omega}\to \cA\stackrel{\Omega}\to \cA$; however, this construction is not functorial, see e.g. \cite[Sec. 1 and Sec. 7]{lurie2012higher}. In order to deal with the non-functoriality of the stabilization one uses Goodwillie calculus to get functorial approximations via Goodwillie-Taylor tower, see e.g. \cite{goodwillie3},  \cite{lurie}, \cite{pereira2013general}.

The non-functoriality of the stabilization is due to the non-equivariance of the morphisms between homotopy theories. This can be understood best in set level constructions. Let $A$ be a finite set on which $\bbN$ acts. Let $\mu_1$ denote the map $\mu_1:A\to A $ given by $\mu_1(a)= 1\cdot a$, where $\cdot $ is the $\bbN$-action. If we universally lift such a $\bbN$ action on $A$ to a $\bbZ$ action (i.e., a  set on which $\bbN$ acts by bijections) then we obtain the largest subset of $A$ on which the restriction of $\mu_1$ is a bijection, see \cite{erdalunlu}. Now, if $f:A\to B$ is a map between finite $\bbN$-sets, then $f$  does not have to induce a map on such maximal subsets unless $f$ is $\bbN$-equivariant. In other words, the association that sends a finite $\bbN$-set to its largest subset on which $\bbN$ acts by bijections does not define an endofunctor on the category of $\bbN$-sets and functions; however, it defines an endofunctor on the category of $\bbN$-sets and $\bbN$-equivariant functions.

 In this paper, we propose another approach to deal with the non-functoriality of the stabilization by considering suitable homotopy theories  equipped with actions of a fixed monoidal category and ``equivariant" functors between them with respect to these actions.  In fact, one can consider  a collection of endofunctors (e.g.,  loop space functors) on a homotopy theory $\cA$ as an action of a monoidal category generated by these functors. Such a homotopy theory is called stable if it is stable under these functors; that is, the monoidal category acts by auto-equivalences. The category of stable homotopy theories with respect to actions of a fixed monoidal category includes in the category of all homotopy theories with actions of the same monoidal category. Then a universal way to assign a stable homotopy theory to a homotopy theory can be given via weak adjoints (i.e., by the reflector and coreflector) of the inclusion of the subcategory of homotopy theories that are stable under the actions. To define these adjoints one merely needs the functors to be equivariant and existence of certain weak (co)limits.  In this setting, one can also easily consider  and construct  stabilizations with extra structures as weak $2$-limits, see Section \ref{sec:concluding}.
 
\subsection*{An outline of the main results}  We discuss our main theme for actions of an enriched monoidal category on objects in a suitably enriched $2$-category. Let $\VV$ be a monoidal category, $\ii$ be an essentially small $\VV$-enriched strict monoidal category and $\frM$ be a $\VCat$-enriched category.  In section \ref{sec:actionsofmonoidalcats}, we discuss $\ii$-actions and  introduce the $\VCat$-enriched category $\ii\frM$ of $\ii$-objects,  $\ii$-equivariant morphisms and $2$-cells  between them.   There is a distinguished full sub-{$\VCat$-category} $\stVM$ of $\ii\frM$  of stable $\ii$-objects, see Definition \ref{def:stable0cell}. Under the condition on existence of certain weak limits (resp. weak colimits), we show that $\stVM$ is coreflective (resp. reflective) in $\ii\frM$, see Theorems \ref{thm:invinftyisstab} and \ref{thm:coinvinftyiscostab}, and Corollary \ref{cor:main1}. These adjoints send  an $\ii$-object  to a stable $\ii$-object; that is, an $\ii$-object on which every object of $\ii$ acts by auto-equivalences.   If $\ii=\bbN$ acts on a relative category $\cA$, where the action is generated by the loop space functor, then the coreflector sends  $\cA$ to to a stable relative category of $\Omega$-spectrum objects in $\cA$ and reflector sends $\cA$ to a stable  relative category akin to  the Spanier-Whitehead category. Therefore, we call the coreflector \emph{stabilization} and the reflector \emph{costabilization}. These adjoints still exist in the lax sense, and in this case they are called \emph{lax stabilization} (which gives prespectra) and  \emph{lax costabilization}. We give constructions of these adjoints in terms of weak (co)ends. In particular, we show that various categories of spectra; such as, sequential spectra, coordinate free spectra, genuine equivariant spectra, parameterized spectra indexed by vector bundles  are obtained as weak ends in a $\VCat$-category. Meanwhile, categories of such prespectra correspond to lax ends. These categories arise formally as consequences of the constructions of relevant  weal $2$-limits  after fixing the  monoidal category, the ambient relative category and the action on it; i.e., after fixing the relative $\ii$-category.

Dually, we observe that the homotopy categories of  costabilizations of relative $\ii$-categories coincide with stable homotopy categories obtained using the constructions of Spanier-Whitehead like categories indexed by $\ii$, see \ref{ssec:laxcostabofrelativecats}.  In other words, ordinary Spanier-Whitehead category and its coordinate free version, equivariant version and parameterized version  are obtained as homotopy categories of certain weak colimits. Thus, we establish that stable homotopy categories of spectra indexed by a symmetric monoidal category $\ii$ are homotopy categories of stabilizations and Spanier-Whitehead like categories indexed by $\ii$ are homotopy categories of costabilizations of pointed $\ii$-homotopy theories, which reveals that homotopy category of spectra and the Spanier-Whitehead category are in fact dual constructions. 

We also show that reduced (co)homology theories, including the ones with exotic gradings, are obtained as objects in an appropriate stabilizations of categories of (co)homology functors with respect to certain monoidal category actions, see Section \ref{ssec:vvgradedtheories}. In particular, ordinary (co)homology theories, representation graded equivariant  (co)homology theories and vector bundles graded parameterized  (co)homology theories  are some of the examples we discuss, see Sections  \ref{sssec:spheregradedCohomologyTheories}, \ref{sssec:ROGradedCohomologyTheories} and \ref{sssec:bundlegradedparameterizedtheories}. This provides a great convenience in determining  the axioms for (co)homology theories with exotic gradings, as they immediately follow as formal consequences of  constructions of certain $2$-limits, and provides a unified framework for (co)homology theories with exotic gradings.

 \subsection*{Organization of the paper}
In Section \ref{sec:actionsofmonoidalcats} we give the definition of $\ii$-actions and the notion of $\ii$-equivariance.  We introduce the $\VCat$-category $\ii\frM$ of $\ii$-objects, $\ii$-equivariant morphisms and $2$-cells between them.  In section \ref{sec:stabilizationcostabilization}, we give definitions and constructions of stabilizations and costabilizations of $\ii$-objects with respect to $\ii$-actions. In section \ref{sec:examplesofstableactions} we show that homology and cohomology theories are objects in the stabilizations of $\ii$-actions on categories of homology and cohomology functors. In section \ref{sec:costabilizationsofrelativecategories} we consider the special case of actions on relative categories and obtain the various categories of spectra as stabilizations, and Spanier-Whitehead like categories as homotopy categories of costabilizations, of relative categories with respect to actions of symmetric monoidal categories. Lastly, we discuss some future directions in Section \ref{ssec:monoidalstructuresrelcat} on stabilization with extra structure; in particular, with an algebraic structure. This includes algebraic structures like symmetric monoidal structures.
\section{Actions of monoidal categories and equivariance}
\label{sec:actionsofmonoidalcats}
 While a collection of endomorphisms of a set generates a monoid, a collection of endofunctors of a category  together with a set of natural transformations between them generates a strict monoidal category.   The motivation for relating monoidal category actions with stabilizations arise from this simple observation.  For a given action of a monoid on a set, one can universally associate another set on which the monoid acts by bijections, see \cite{erdalunlu}. The stabilization of a homotopy theory, on the other hand, coincides with the categorification of this idea since it is obtained by formally inverting all suspension  (or loop space) functors. 
 
 The suspension or loop space functor are often considered as enriched functors over a monoidal category. Then the  monoidal category generated by these functors is also enriched.
 In this section, we present an enriched categorification of the definitions of monoid actions on sets and equivariant functions given in \cite{erdalunlu}. In particular, we define the action of an enriched monoidal category on an object in a enriched $2$-category and explore the properties of this notion.

\subsection*{Conventions on categories} \label{convention} We assume the axiom of Gr\"othendieck universes, \cite{dhks}. For a universe $\mathcal U$, when a monoidal category $\ii$ acts on an object $\cA$, we assume $\ii$ is $\mathcal U$-small and  $\cA$   belongs to a $\mathcal U^+$-category. In particular, if $\cA$ is a category itself, then $\cA$ will be a $\mathcal  U$-category belonging to a $\mathcal U^+$-category of $\mathcal U$-categories. Thus, by a small category we mean it is $\mathcal  U$-small with respect to a fixed universe $\mathcal  U$.

\subsection{{$\ii$}-actions}
\label{ssec:vvactions}
Let $(\VV,\otimes,\vunit)$ be a symmetric monoidal category and $\VCat$ denote the cartesian closed symmetric monoidal category of $\VV$-enriched categories (or simply $\VV$-categories). The unit of $\VCat$ is denoted by $\cunit$. We assume that the forgetful functor $\VCat\to \Cat$ is faithful. This, for instance, holds for $\VV=\Set,\ \cT,\ \cAb,\ \SP,$ etc. A category enriched in $\VCat$ will be called $\VCat$-category, $\VCat$-enriched functor will be called $\VCat$-functor and a $\VCat$-enriched natural transformation will be called $\VCat$-natural transformation. If $\VV=\Set$, the category of sets, then they will be just be called $2$-category, $2$-functor and $2$-natural transformation, respectively.

Let $(\ii,{\circledast}, {\unit} )$ be an essentially small $\VV$-enriched strict monoidal category. There is an associated $\VCat$-category $\mathbf{B}\ii$ with a single object $*$, called the delooping $\VCat$-category of $\ii$. The morphisms (i.e., $1$-morphisms) from $*$ to $*$ are objects of $\ii$ and the composition of two morphisms in $\mathbf{B}\ii$ is given by the monoidal product on  $\ii$. We write $\mathbf{B}\ii^{\op}$ for morphism dual of $\mathbf{B}\ii$; that is, the $\VCat$-category obtained by reversing its morphisms but not the $2$-cells. For the rest of the paper we use $\HH$ for the $\VCat$-category $\mathbf{B}\ii\times \mathbf{B}\ii^{\op}$. Let $\frM$ be a $\VCat$-category and $\cA$ be a object in $\frM$. An \emph{action of $\ii $ on $\cA$} (or a \emph{$\ii $-action on $\cA$}) is a strict $\VCat$-functor \[\alpha :\HH\rightarrow \frM\] where $\alpha ({*,*})=\cA$.

This definition is, in fact, the same as the definition of ordinary biaction. However, we construct a $\VCat$-category equipped with an exotic notion of $\ii$-equivariance, see Section \ref{ssec:vvequivariant1cells}, so that the category whose morphisms are $\ii$-equivariant will be different than the category of objects with bi-actions.

\begin{remark}\label{rem:pseudotostrict} Note that for $\VV=\Set$, we can define actions of a non-strict monoidal category  after strictification. In other words, if $\ii $ is a monoidal category (not necessarily strict) we can define a $\ii $-action as a $\mathrm{str}(\ii)$-action where $\mathrm{str}(\ii)$ denotes the strictification of the monoidal category $\ii$. Equivalently, for a arbitrary monoidal category $\ii$ we can  construct $\HH $ as above and any pseudofunctor from $\HH $ to a $2$-category can be considered as an $\ii$-action by the strictification adjunction given in \cite{campbell2019}.
\end{remark}  
Clearly,  ordinary actions of monoids are trivial examples.  Some examples where higher morphisms are more interesting are given in \cite{kelly2001} which can be considered as examples to our actions by composing with the projection from $\mathbf{B}\ii\times \mathbf{B}\ii^{\op}$ onto its left component $\mathbf{B}\ii$ and considering Remark \ref{rem:pseudotostrict}.

\begin{example}\label{ex:tensoraction}
	  If $\ii$ is a strict monoidal category, then $\ii$ (or any of its monoidal subcategory) acts on $\ii$ by its tensor product. One can define an action by tensoring from the left, or from the right or from both sides. More generally, if $\cA$ is $\ii$-enriched and (co)powered over $\ii$, then the (co)power defines an $\ii$-action on $\cA$.
\end{example}
\begin{example}\label{ex:1-cellgeneratedactiongeneratedaction}
Any $1$-endomorphism on a object $\cA$ in a $\VCat$-category ${\frM}$ generates a $\bbN$ action by considering $\bbN$ as a monoidal category with only identity morphisms. More generally, one can choose a collection of $1$-endomorphisms of $\cA$ and take the monoidal category generated by these endomorphisms. The resulting monoidal category acts on $\cA$. 
\end{example}

Some examples that have interest in homotopy theory are discussed in Section \ref{ssec:knownexamplesofvvgradedtheories}. 
\begin{notation}
	We use the notation $\cA_{\alpha}$ for a object $\cA$ with $\ii$-action $\alpha$.
\end{notation}

\subsection{Enriched functors and weak (co)ends}
\label{ssec:enrichedlaxpseudo}
Since the forgetful functor $\VCat\to \Cat$ is faithful, the $\VCat$-enriched weak (co)limits in this paper are defined in accordance with the notion of bilimits  in \cite[2.6]{birdflexiblelimits}. The weak ends, in particular, can be identified with hom-weighted weak limits. 
Let $F_0:\cC_0^{\op}\times \cC_0\to \frT$ a strict $2$-functor between  $2$-categories. Then the pseudo-end $\oint_{x:\cC}F({x,x})$ is an object in $\frT$ defined by the isomorphism
$$ 
\frT({Z},\oint_{x:\cC_0}F_0({x,x}))\cong \Psd(\cC_0^{\op}\times \cC_0,\Cat)\left( \hom_{\cC_0}(-,-), \frT({Z},F_0(-,-))\right) 
$$
that is natural in ${Z}$, where $\frT(-,-)$ denotes the hom-category in $\frT$. Here $ \Psd(\cC_0^{\op}\times \cC_0,\Cat)$ denotes the $2$-category of strict $2$-functors, pseudo-natural transformations and modifications. 
Similarly, the lax-end $\sqint_{x:\cC}F({x,x})$ is defined by replacing $\Psd$ with $\Lax$, the $2$-category of $2$-functors, lax natural transformations and modifications.

Let $\cC$ be a $\VCat$-category. For  $\VCat$-enriched functors $H,K:\cC\to\VCat$, the $\VV$-category $\Psd^\VV(\cC,\VCat)(H,K)$ is defined by $$\Psd^\VV(\cC,\VCat)(H,K)=\oint_{x:\cC_0}\VCat(H(x),K(x)).$$
Here $\VCat(-,-)$ denotes the internal hom in $\VCat.$ We obtain $\Lax^\VV(\cC,\VCat)(H,K)$ by replacing the pseudo-end with lax-end.
 
Let   $\frM$ be a $\VCat$-enriched category. Throughout the paper, by an equivalence between $\VV$-categories we mean an adjoint equivalence. We have a functor from $\frM^{\op }\times \frM$ to $\VCat $ which sends $({\cA},{\cB})$ to the $\VV$-category $\frM({\cA},{\cB})$, which we simply denote by $[{\cA},{\cB}]$. Given a $\VCat$-functor $F:\cC^{\op}\times \cC\to \frM$ between $\VCat$-categories, its weak end $\oint_{x:\cC}F({x,x})$ is defined by the equivalence in the $2$-category $\VCat$
$$ 
[{Z},\oint_{x:\cC}F({x,x})] \simeq \Psd^\VV(\cC^{\op}\times \cC,\VCat)\left( \hom_\cC(-,-), [{Z},F(-,-)]\right) 
$$
natural in ${Z}$. Similarly, the lax end $\sqint_{x:\cC}F({x,x})$ is defined by replacing $\Psd^\VV$ with $\Lax^\VV$. In other words, the weak end of $F$ is a $\hom_\cC$ weighted weak limit of $F$ (see also \cite[Defn. 5.1.11]{johnson20212}).

\subsection{$\VCat$-enriched weak end of an $\ii$-action}
\label{ssec:vcatenrichedends}
Let $\cA$ be a object in the $\VCat$-category $\frM$, and $\alpha:\HH\rightarrow \frM$ be an $\ii$-action on $\cA$. Since $\mathbf{B}\ii$ has a single object, The underlying category of the $\VV$-category
\[ \Psd^\VV(\HH,\VCat)\left( \hom_{\mathbf{B}\ii^{\op}}(-,-), [{Z},\alpha(-,-)]\right) \]
on the right hand side of the equivalence that defines the weak end
\[ {\oint}_{x:\mathbf{B}\ii^{\op}}\alpha({x,x}) \] is equivalent to the category whose objects are morphisms
\[\omega:{Z}\rightarrow \cA\] 
and invertible $2$-cells (i.e., a $2$-isomorphisms) \[\sigma({u}):\alpha({\unit}, {u}^{\op})\circ \omega \Rightarrow \alpha({u},{\unit}^{\op})\circ \omega\] assigned to all objects ${u}$ in $\ii$ so that
we have
	\begin{equation}\label{eqn:naturalityofsigma} \sigma({v})\bullet (\alpha(\id_{\unit},f) \rhd \omega) = (\alpha(f,\id_{{\unit} ^{\op}}) \rhd \omega)\bullet \sigma ({u})
	\end{equation}
for every morphism $f:{u}\rightarrow {v} $ in $\ii $ and
the following diagram of $2$-cells
	\begin{equation}\label{diag:triangleofend}\begin{tikzpicture}[scale=0.8]
	\node (A) at  (1,1) {$  \alpha({\unit},{v}^{\op})\circ\alpha({\unit},{u}^{\op})\circ \omega$ };
	\node (B) at (4.5,4) {$\alpha({\unit},{v}^{\op})\circ\alpha({u},{\unit}^{\op})\circ \omega = \alpha({u},{\unit}^{\op})\circ\alpha({\unit},{v}^{\op})\circ \omega$};
	\node (D) at (8,1) {$\alpha({u},{\unit}^{\op})\circ\alpha({v},{\unit}^{\op})\circ \omega$};
	\node (E) at (1,0) {$\alpha({\unit},({u}{\circledast} {v})^{\op})\circ \omega$};
	\node (G) at (8,0) {$\alpha({u}{\circledast}  {v},{\unit}^{\op})\circ \omega$};
\draw[-Implies,line width=.5pt,double distance=2pt]	(E) -- node[below]{$\sigma({u}{\circledast} {v})$} (G);
\draw[-Implies,line width=.5pt,double distance=2pt]	(A) -- node[left]{$\alpha({\unit},{v}^{\op})(\sigma({u})) \ \ $} (B);
\draw[-=,line width=.5pt,double distance=2pt]	(B) -- node[above]{} (B);
\draw[-Implies,line width=.5pt,double distance=2pt]	(B) -- node[right]{$ \ \ \alpha({u},{\unit}^{\op})(\sigma({v}))$} (D);
\draw[-=,line width=.5pt,double distance=2pt]	(D) -- node[right]{} (G);
\draw[-=,line width=.5pt,double distance=2pt]	(E) -- node[right]{} (A);
	\end{tikzpicture}\end{equation}
commutes for every ${u}$, ${v}$ in $\ii$ and
\begin{equation}\label{eqn:sigma1}
	\id_{\alpha({\unit},{\unit}^{\op })\circ \omega}=\sigma({\unit}).
\end{equation}
Morphisms from $(\omega ,\sigma)$ to $(\overline{\omega},\overline{\sigma})$ are $2$-cells $\theta $ from the morphism $\omega$ to the morphism $\overline{\omega}$ so that
\[ \overline{\sigma}(u) \bullet (\alpha({\unit} ,u^{\op}) \lhd \theta) = (\alpha(u,{\unit} ^{\op}) \lhd \theta)\bullet \sigma(u) \]
for all objects ${u}$ in $\ii$.  
	\begin{remark}
	In the case when $\frM$ has an object $*$ such that $[*,\cA]\cong \cA$ for every object $\cA$, than one obtains that \[ {\oint}_{x:\mathbf{B}\ii^{\op}}\alpha({x,x}) \] has objects as pairs $(a,\sigma)$ where $a$ is a morphism $*\to\cA$ and $\sigma$ is the natural transformation as above that satisfies \ref{eqn:naturalityofsigma}, \ref{diag:triangleofend} and \ref{eqn:sigma1}. A morphism from $(a ,\sigma)$ to $(\overline{a} ,\overline{\sigma})$ is a morphism  $\theta:a\to  \overline{a}$ in $\cA$ that satisfies \[ \overline{\sigma}(u) \bullet (\alpha({\unit} ,u^{\op}) \lhd \theta) = (\alpha(u,{\unit} ^{\op}) \lhd \theta)\bullet \sigma(u).\] If  $\frM=\CAT$, that is $\cA$ is a category, we can choose $Z$ as the terminal category $*$.
	\end{remark}

\subsection{{$\ii$}-equivariant {$1$}-cells}
\label{ssec:vvequivariant1cells}
Let \begin{equation}\label{eqn:psi}{\psi}:\HH   \longrightarrow
{\HH}^{\op}\times {\HH}\end{equation} 
be the $\VCat$-functor given by $\psi({u},{v}^{\op})=(({v},{u}^{\op})^{\op},({u},{v}^{\op}))$. Given two $\ii$-actions $\alpha $ and $\beta$ on objects ${\cA}_{\alpha}$ and ${\cB}_{\beta}$, we define an $\ii$-action on $[{\cA}_{\alpha},{\cB}_{\beta}]$ by the composition
\begin{equation}\label{eqn:actiononAB}\HH  \stackrel{\psi}\rightarrow
{\HH}^{\op}\times {\HH} \stackrel{(\alpha^{\op}, \beta)}\longrightarrow
\frM^{\op}\times \frM\stackrel{[- ,- ]}\longrightarrow \CAT .\end{equation} 
We denote this $\ii$-action by $[\alpha ,\beta]$. We define the category $\Fun_\ii ({\cA}_{\alpha},{\cB}_{\beta})$ of \emph{$\ii$-equivariant morphisms} from ${\cA}_{\alpha}$ to ${\cB}_{\beta}$ as the end  
\[\Fun_{\ii }({\cA}_{\alpha},{\cB}_{\beta})={\oint}_{x:\mathbf{B}\ii^{\op}}[\alpha,\beta]{(x,x) }\] provided that it exists. Any object in $\Fun_{\ii }({\cA}_{\alpha},{\cB}_{\beta})$ is called {$\ii$-equivariant morphism} from ${\cA}_{\alpha}$ to ${\cB}_{\beta}$.
The category $\Fun_\ii ({\cA}_{\alpha},{\cB}_{\beta})$ will in fact be considered as the hom-category of a $\VCat$-category. Therefore, we give an explicit  description of $\Fun_\ii ({\cA}_{\alpha},{\cB}_{\beta})$ that is functorial on $({\cA}_{\alpha},{\cB}_{\beta})$. Objects of  $\Fun_\ii ({\cA}_{\alpha},{\cB}_{\beta})$ are pairs $(f,\tau)$ such that $f$ is an object in $[{\cA}_{\alpha},{\cB}_{\beta}]$ and $\tau$ is a natural isomorphism from  
$[\alpha,\beta]({\unit},-^{\op})(f):\ii\to [{\cA}_{\alpha},{\cB}_{\beta}]$ to $  [\alpha,\beta](-,{\unit}^{\op})(f):\ii\to [{\cA}_{\alpha},{\cB}_{\beta}]$. We define $\omega(f,\tau)=f$ and  for every object ${u}$ in $\ii$, we define the $(f,\tau)$ component of the $2$-cell $\sigma({u})$   as $\tau({u})$. So that the above compatibility conditions are satisfied. In particular, given a
$\ii$-equivariant  morphism $(f,\tau)$ from ${\cA}_{\alpha}$ to ${\cB}_{\beta}$  we have an  morphism
\[\sigma({u})_{(f,\tau)}=\tau({u}):\beta({\unit},{u}^{\op})\circ f  \circ \alpha({u},{\unit}^{\op}) \rightarrow \beta({u},{\unit}^{\op})\circ  f  \circ \alpha({\unit},{u}^{\op}) \] in $[{\cA}_{\alpha},{\cB}_{\beta}]$ for every ${u}$ in $\ii$.

\subsection{The $\VCat$-category of $\ii$-actions}
\label{ssec:2categoryvvM}
We define two functors $\pi_l,\,\pi_r :\HH\to \HH $ as  $\pi_l:({u},{v}^{\op})\mapsto ({u},{\unit}^{\op})$ and  $\pi_r:({u},{v}^{\op})\mapsto ({\unit},{v}^{\op})$ on morphisms and corresponding projections on $2$-cells. We say a $\ii $-action $\alpha $ is a right (respectively left) action if $\alpha\circ \pi _r=\alpha $ (respectively $\alpha\circ \pi _l=\alpha $). An action will be called an action mute on one side if it is a right action or a left action. We will say a left action is mute on right and a right action is mute on left. Now we start defining the $\VCat$-category $\ii \frM$. The objects of $\ii \frM$ are the objects of $\frM$ equipped with a $\ii $-action mute on at least one side. The morphisms of $\ii \frM$ are $\ii$-equivariant morphisms as defined in the previous section. Now we define the composition of two equivariant morphisms.

Let   $\cA_{\alpha}$, $\cB_{\beta}$, and $\cC_{\gamma}$ be objects in $\ii \frM$. Assume  $(f,\tau_f)$ is an $\ii$-equivariant morphism from $\cA_{\alpha}$ to $\cB_{\beta}$ and $(g,\tau_g)$ is an $\ii$-equivariant morphism from $\cB_{\beta}$ to $\cC_{\gamma}$. We define the composition of  $(f,\tau_f)$ and $(g,\tau_g)$ as the pair  $(g\circ f,\tau_{g\circ f})$  where   $\tau_{g\circ f}({u})$ for ${u}$ in $\ii$ is defined as follows:  In case $\beta $ is mute on left and hence a right action, for ${u}$ in $\ii$, we have $\beta({u},{\unit}^{\op})=\id_\cB$,  so we can define $\tau_{g\circ f}({u})$ as the isomorphism given by the vertical composition of the isomorphisms:
\[\begin{tikzpicture}[scale=1]
	\node (71) at  (0,6) {$\cA$ };
	\node (72) at (2,6) {$\cA$};
	\node (73) at (4,6) {$\cB$};
	\node (74) at (6,6) {$\cB$};
	\node (75) at (8,6) {$\cC$};
	\node (76) at (10,6) {$\cC$};

	\node (61) at  (0,4.5) {$\cA$ };
	\node (62) at (2,4.5) {$\cA$};
	\node (63) at (4,4.5) {$\cB$};
	\node (64) at (6,4.5) {$\cB$};
	\node (65) at (8,4.5) {$\cC$};
	\node (66) at (10,4.5) {$\cC$};

	\node (51) at  (0,3) {$\cA$ };
	\node (52) at (2,3) {$\cA$};
	\node (53) at (4,3) {$\cB$};
	\node (54) at (6,3) {$\cB$};
	\node (55) at (8,3) {$\cC$};
	\node (56) at (10,3) {$\cC$};

	\path[->,font=\scriptsize,>=angle 90]

	(71) edge node[above]  {${\alpha({u},{\unit}^{\op})}$} (72)
	(72) edge node[above]  {$ f $} (73)
	(73) edge node[above]  {${\beta({u},{\unit}^{\op})}$} (74)
	(74) edge node[below] (3g) {} node[above]  {$ g $} (75)
	(75) edge node[above]  {${\gamma({\unit},{u}^{\op})}$} (76)

	(61) edge node[above]  {${\alpha({u},{\unit}^{\op})}$} (62)
	(62) edge node[below] (3f) {} node[above]  {$ f $} (63)
	(63) edge node[above]  {${\beta({\unit},{u}^{\op})}$} (64)
	(64) edge   node[above] (2g) {$ g $} (65)
	(65) edge node[above]  {${\gamma({u},{\unit}^{\op})}$} (66)

	(51) edge node[above]  {${\alpha({\unit},{u}^{\op})}$} (52)
	(52) edge node[above]  (2f) {$ f $} (53)
	(53) edge node[above]  {${\beta({u},{\unit}^{\op})}$} (54)
	(54) edge node[above] {$ g $} (55)
	(55) edge node[above]  {${\gamma({u},{\unit}^{\op})}$} (56)

 	(72) edge[bend left] node [above] {$g\circ f$} (75)
(52) edge[bend right] node [below] {${g\circ f}$} (55);

	\draw[-=,line width=.5pt,double distance=2pt] (51) -- node[right]{} (61);
	
	
	\draw[-=,line width=.5pt,double distance=2pt] (54) -- node[right]{} (64); 
	\draw[-=,line width=.5pt,double distance=2pt] (55) -- node[right]{} (65); 
	\draw[-=,line width=.5pt,double distance=2pt] (56) -- node[right]{} (66);

	\draw[-=,line width=.5pt,double distance=2pt] (61) -- node[right]{} (71); 
	\draw[-=,line width=.5pt,double distance=2pt] (62) -- node[right]{} (72); 
	\draw[-=,line width=.5pt,double distance=2pt] (63) -- node[right]{} (73); 
	
	
	\draw[-=,line width=.5pt,double distance=2pt] (66) -- node[right]{} (76);


	\draw[-{Implies},double distance=3pt,shorten >=2pt,shorten <=2pt,font=\scriptsize,] (3g) to node[scale=1] [right] {$\tau_{g}({u})$} (2g);
	\draw[-{Implies},double distance=3pt,shorten >=2pt,shorten <=2pt,font=\scriptsize,] (3f) to node[scale=1] [right] {$\tau_{f}({u})$} (2f);
\end{tikzpicture}\]
 In case $\beta $ is mute on the right and hence a left action, for ${u}$ in $\ii$ we have  $\beta({\unit} , {u}^{\op})=\id_\cB$ and so we can define $\tau_{g\circ f}({u})$ as the isomorphism given by the vertical composition of the following isomorphisms:
 \[\begin{tikzpicture}[scale=1]

 	\node (71) at  (0,6) {$\cA$ };
 	\node (72) at (2,6) {$\cA$};
 	\node (73) at (4,6) {$\cB$};
 	\node (74) at (6,6) {$\cB$};
 	\node (75) at (8,6) {$\cC$};
 	\node (76) at (10,6) {$\cC$};

 	\node (61) at  (0,4.5) {$\cA$ };
 	\node (62) at (2,4.5) {$\cA$};
 	\node (63) at (4,4.5) {$\cB$};
 	\node (64) at (6,4.5) {$\cB$};
 	\node (65) at (8,4.5) {$\cC$};
 	\node (66) at (10,4.5) {$\cC$};

 	\node (51) at  (0,3) {$\cA$ };
 	\node (52) at (2,3) {$\cA$};
 	\node (53) at (4,3) {$\cB$};
 	\node (54) at (6,3) {$\cB$};
 	\node (55) at (8,3) {$\cC$};
 	\node (56) at (10,3) {$\cC$};

 	\path[->,font=\scriptsize,>=angle 90]

 	(71) edge node[above]  {${\alpha({u},{\unit}^{\op})}$} (72)
 	(72) edge node[below] (3f) {} node[above]  {$ f $} (73)
 	(73) edge node[above]  {${\beta({\unit},{u}^{\op})}$} (74)
 	(74) edge node[above]  {$ g $} (75)
 	(75) edge node[above]  {${\gamma({\unit},{u}^{\op})}$} (76)

 	(61) edge node[above]  {${\alpha({u},{\unit}^{\op})}$} (62)
 	(62) edge node[above] (2f)  {$ f $} (63)
 	(63) edge node[above]  {${\beta({u},{\unit}^{\op})}$} (64)
 	(64) edge node[below] (3g) {} node[above] {$ g $} (65)
 	(65) edge node[above]  {${\gamma({u},{\unit}^{\op})}$} (66)

 	(51) edge node[above]  {${\alpha({\unit},{u}^{\op})}$} (52)
 	(52) edge node[above]  {$ f $} (53)
 	(53) edge node[above]  {${\beta({\unit},{u}^{\op})}$} (54)
 	(54) edge node[above] (2g) {$ g $} (55)
 	(55) edge node[above]  {${\gamma({u},{\unit}^{\op})}$} (56)

 	(72) edge[bend left] node [above] {${g\circ f}$} (75)
 	(52) edge[bend right] node [below] {${g\circ f}$} (55)

 	;

 	\draw[-=,line width=.5pt,double distance=2pt] (51) -- node[right]{} (61);
 	
 	\draw[-=,line width=.5pt,double distance=2pt] (52) -- node[right]{} (62); 
 	\draw[-=,line width=.5pt,double distance=2pt] (53) -- node[right]{} (63); 
 	
 	\draw[-=,line width=.5pt,double distance=2pt] (56) -- node[right]{} (66);

 	\draw[-=,line width=.5pt,double distance=2pt] (61) -- node[right]{} (71); 
 	
 	\draw[-=,line width=.5pt,double distance=2pt] (64) -- node[right]{} (74); 
 	\draw[-=,line width=.5pt,double distance=2pt] (65) -- node[right]{} (75); 
 	
 	\draw[-=,line width=.5pt,double distance=2pt] (66) -- node[right]{} (76);


 	\draw[-{Implies},double distance=3pt,shorten >=2pt,shorten <=2pt,font=\scriptsize,] (3g) to node[scale=1] [right] {$\tau_{g}({u})$} (2g);
 	\draw[-{Implies},double distance=3pt,shorten >=2pt,shorten <=2pt,font=\scriptsize,] (3f) to node[scale=1] [right] {$\tau_{f}({u})$} (2f);
 \end{tikzpicture}\]
 
Notice in case $\beta $ is mute on both sides, these definitions coincide by the interchange law. 

Let   $\cA_{\alpha}$, $\cB_{\beta}$, $\cC_{\gamma}$ and $\cD_\delta$ be objects in $\ii \frM$ and  $(f,\tau_f):\cA_{\alpha}\to \cB_{\beta}$, $(g,\tau_g):\cB_{\beta}\to \cC_{\gamma}$ and $(h,\tau_h):\cC_{\gamma}\to \cD_\delta$ be $\ii$-equivariant morphisms. In the case when $\beta$ and $\gamma$  are both  mute on the left, the associativity can be obtained from the following diagram:
\[\begin{tikzpicture}[scale=.9]
		\node (71) at  (0,6) {$\cA$ };
	\node (72) at (2,6) {$\cA$};
	\node (73) at (4,6) {$\cB$};
	\node (74) at (6,6) {$\cB$};
	\node (75) at (8,6) {$\cC$};
	\node (76) at (10,6) {$\cC$};
	\node (77) at (12,6) {$\cD$};
	\node (78) at (14,6) {$\cD$};

	\node (61) at  (0,4.5) {$\cA$ };
	\node (62) at (2,4.5) {$\cA$};
	\node (63) at (4,4.5) {$\cB$};
	\node (64) at (6,4.5) {$\cB$};
	\node (65) at (8,4.5) {$\cC$};
	\node (66) at (10,4.5) {$\cC$};
	\node (67) at (12,4.5) {$\cD$};
	\node (68) at (14,4.5) {$\cD$};
	
	\node (51) at  (0,3) {$\cA$ };
	\node (52) at (2,3) {$\cA$};
	\node (53) at (4,3) {$\cB$};
	\node (54) at (6,3) {$\cB$};
	\node (55) at (8,3) {$\cC$};
	\node (56) at (10,3) {$\cC$};
	\node (57) at (12,3) {$\cD$};
	\node (58) at (14,3) {$\cD$};

	\node (41) at  (0,1.5) {$\cA$ };
	\node (42) at (2,1.5) {$\cA$};
	\node (43) at (4,1.5) {$\cB$};
	\node (44) at (6,1.5) {$\cB$};
	\node (45) at (8,1.5) {$\cC$};
	\node (46) at (10,1.5) {$\cC$};
	\node (47) at (12,1.5) {$\cD$};
	\node (48) at (14,1.5) {$\cD$};

	\path[->,font=\scriptsize,>=angle 90]

	(71) edge node[above]  {${\alpha({u},{\unit}^{\op})}$} (72)
	(72) edge node[above]  {$ f $} (73)
	(73) edge node[above]  {${\beta({u},{\unit}^{\op})}$} (74)
	(74) edge node[above]  {$ g $} (75)
	(75) edge node[above]  {${\gamma({u},{\unit}^{\op})}$} (76)
	(76) edge node[below] (4h) {} node[above]  {$ h $} (77)
	(77) edge node[above]  {${\delta({\unit},{u}^{\op})}$} (78)

	(61) edge node[above]  {${\alpha({u},{\unit}^{\op})}$} (62)
	(62) edge node[above]  {$ f $} (63)
	(63) edge node[above]  {${\beta({u},{\unit}^{\op})}$} (64)
	(64) edge node[below] (3g) {} node[above]  {$ g $} (65)
	(65) edge node[above]  {${\gamma({\unit},{u}^{\op})}$} (66)
	(66) edge node[above] (3h) {$ h $} (67)
	(67) edge node[above]  {${\delta({u},{\unit}^{\op})}$} (68)

	(51) edge node[above]  {${\alpha({u},{\unit}^{\op})}$} (52)
	(52) edge node[below] (3f) {} node[above]  {$ f $} (53)
	(53) edge node[above]  {${\beta({\unit},{u}^{\op})}$} (54)
	(54) edge node[above] (2g) {$ g $} (55)
	(55) edge node[above]  {${\gamma({u},{\unit}^{\op})}$} (56)
	(56) edge node[above]  {$ h $} (57)
	(57) edge node[above]  {${\delta({u},{\unit}^{\op})}$} (58)

	(41) edge node[above]  {${\alpha({\unit},{u}^{\op})}$} (42)
	(42) edge node[above] (2f) {$ f $} (43)
	(43) edge node[above]  {${\beta({u},{\unit}^{\op})}$} (44)
	(44) edge node[above]  {$ g $} (45)
	(45) edge node[above]  {${\gamma({u},{\unit}^{\op})}$} (46)
	(46) edge node[above]  {$ h $} (47)
	(47) edge node[above]  {${\delta({u},{\unit}^{\op})}$} (48)
	
	(72) edge[bend left=30] node [above] {$h\circ (g\circ f)=(h\circ g)\circ f$} (77) 
	(42) edge[bend right=30] node [below] {$h\circ (g\circ f)=(h\circ g)\circ f$} (47) 
	;

	\draw[-=,line width=.5pt,double distance=2pt] (41) -- node[right]{} (51); 
	\draw[-=,line width=.5pt,double distance=2pt] (51) -- node[right]{} (61); 
	\draw[-=,line width=.5pt,double distance=2pt] (61) -- node[right]{} (71);

	\draw[-=,line width=.5pt,double distance=2pt] (52) -- node[right]{} (62); 
	\draw[-=,line width=.5pt,double distance=2pt] (62) -- node[right]{} (72);

	\draw[-=,line width=.5pt,double distance=2pt] (53) -- node[right]{} (63); 
	\draw[-=,line width=.5pt,double distance=2pt] (63) -- node[right]{} (73);

	\draw[-=,line width=.5pt,double distance=2pt] (44) -- node[right]{} (54); 
	\draw[-=,line width=.5pt,double distance=2pt] (64) -- node[right]{} (74); 

	\draw[-=,line width=.5pt,double distance=2pt] (45) -- node[right]{} (55); 
	\draw[-=,line width=.5pt,double distance=2pt] (65) -- node[right]{} (75);

	\draw[-=,line width=.5pt,double distance=2pt] (46) -- node[right]{} (56); 
	\draw[-=,line width=.5pt,double distance=2pt] (56) -- node[right]{} (66); 

	\draw[-=,line width=.5pt,double distance=2pt] (47) -- node[right]{} (57); 
	\draw[-=,line width=.5pt,double distance=2pt] (57) -- node[right]{} (67); 

	\draw[-=,line width=.5pt,double distance=2pt] (48) -- node[right]{} (58); 
	\draw[-=,line width=.5pt,double distance=2pt] (58) -- node[right]{} (68); 
	\draw[-=,line width=.5pt,double distance=2pt] (68) -- node[right]{} (78); 
	
	\draw[-{Implies},double distance=3pt,shorten >=2pt,shorten <=2pt,font=\scriptsize,] (4h) to node[scale=1] [right] {$\tau_h({u})$} (3h);
	\draw[-{Implies},double distance=3pt,shorten >=2pt,shorten <=2pt,font=\scriptsize,] (3g) to node[scale=1] [right] {$\tau_{g}({u})$} (2g);
	\draw[-{Implies},double distance=3pt,shorten >=2pt,shorten <=2pt,font=\scriptsize,] (3f) to node[scale=1] [right] {$\tau_{f}({u})$} (2f);
\end{tikzpicture}\]

If $\beta$ is mute on the left and $\gamma$ is mute on the right then $\tau_{h\circ (g\circ f)}({u})$ is obtained from the following diagram 
 \[\begin{tikzpicture}[scale=.9]	\node (71) at  (0,6) {$\cA$ };
	\node (72) at (2,6) {$\cA$};
	\node (73) at (4,6) {$\cB$};
	\node (74) at (6,6) {$\cB$};
	\node (75) at (8,6) {$\cC$};
	\node (76) at (10,6) {$\cC$};
	\node (77) at (12,6) {$\cD$};
	\node (78) at (14,6) {$\cD$};

	\node (61) at  (0,4.5) {$\cA$ };
	\node (62) at (2,4.5) {$\cA$};
	\node (63) at (4,4.5) {$\cB$};
	\node (64) at (6,4.5) {$\cB$};
	\node (65) at (8,4.5) {$\cC$};
	\node (66) at (10,4.5) {$\cC$};
	\node (67) at (12,4.5) {$\cD$};
	\node (68) at (14,4.5) {$\cD$};
	
	\node (51) at  (0,3) {$\cA$ };
	\node (52) at (2,3) {$\cA$};
	\node (53) at (4,3) {$\cB$};
	\node (54) at (6,3) {$\cB$};
	\node (55) at (8,3) {$\cC$};
	\node (56) at (10,3) {$\cC$};
	\node (57) at (12,3) {$\cD$};
	\node (58) at (14,3) {$\cD$};

	\node (41) at  (0,1.5) {$\cA$ };
	\node (42) at (2,1.5) {$\cA$};
	\node (43) at (4,1.5) {$\cB$};
	\node (44) at (6,1.5) {$\cB$};
	\node (45) at (8,1.5) {$\cC$};
	\node (46) at (10,1.5) {$\cC$};
	\node (47) at (12,1.5) {$\cD$};
	\node (48) at (14,1.5) {$\cD$};

	\path[->,font=\scriptsize,>=angle 90]

	(71) edge node[above]  {${\alpha({u},{\unit}^{\op})}$} (72)
	(72) edge node[above]  {$ f $} (73)
	(73) edge node[above]  {${\beta({u},{\unit}^{\op})}$} (74)
	(74) edge node[below] (3g) {} node[above]  {$ g $} (75)
	(75) edge node[above]  {${\gamma({\unit},{u}^{\op})}$} (76)
	(76) edge node[above]  {$ h $} (77)
	(77) edge node[above]  {${\delta({\unit},{u}^{\op})}$} (78)

	(61) edge node[above]  {${\alpha({u},{\unit}^{\op})}$} (62)
	(62) edge node[below] (3f) {} node[above]  {$ f $} (63)
	(63) edge node[above]  {${\beta({\unit},{u}^{\op})}$} (64)
	(64) edge   node[above] (2g) {$ g $} (65)
	(65) edge node[above]  {${\gamma({u},{\unit}^{\op})}$} (66)
	(66) edge  node[above] {$ h $} (67)
	(67) edge node[above]  {${\delta({\unit},{u}^{\op})}$} (68)

	(51) edge node[above]  {${\alpha({\unit},{u}^{\op})}$} (52)
	(52) edge node[above]  (2f) {$ f $} (53)
	(53) edge node[above]  {${\beta({u},{\unit}^{\op})}$} (54)
	(54) edge node[above] {$ g $} (55)
	(55) edge node[above]  {${\gamma({u},{\unit}^{\op})}$} (56)
	(56) edge node[below] (4h) {} node[above] {$ h $} (57)
	(57) edge node[above]  {${\delta({\unit},{u}^{\op})}$} (58)

	(41) edge node[above]  {${\alpha({\unit},{u}^{\op})}$} (42)
	(42) edge node[above] {$ f $} (43)
	(43) edge node[above]  {${\beta({u},{\unit}^{\op})}$} (44)
	(44) edge node[above]  {$ g $} (45)
	(45) edge node[above]  {${\gamma({\unit},{u}^{\op})}$} (46)
	(46) edge node[above]   (3h) {$ h $} (47)
	(47) edge node[above]  {${\delta({u},{\unit}^{\op})}$} (48)

	(72) edge[bend left=30] node [above] {$h\circ (g\circ f)$} (77) 
	
	(42) edge[bend right=30] node [below] {$h\circ (g\circ f)$} (47) 
	;

	\draw[-=,line width=.5pt,double distance=2pt] (41) -- node[right]{} (51); 
	\draw[-=,line width=.5pt,double distance=2pt] (51) -- node[right]{} (61); 
	\draw[-=,line width=.5pt,double distance=2pt] (61) -- node[right]{} (71);

	\draw[-=,line width=.5pt,double distance=2pt] (42) -- node[right]{} (52); 
	\draw[-=,line width=.5pt,double distance=2pt] (62) -- node[right]{} (72);

	\draw[-=,line width=.5pt,double distance=2pt] (43) -- node[right]{} (53); 
	\draw[-=,line width=.5pt,double distance=2pt] (63) -- node[right]{} (73);

	\draw[-=,line width=.5pt,double distance=2pt] (44) -- node[right]{} (54); 
	\draw[-=,line width=.5pt,double distance=2pt] (54) -- node[right]{} (64); 

	\draw[-=,line width=.5pt,double distance=2pt] (45) -- node[right]{} (55); 
	\draw[-=,line width=.5pt,double distance=2pt] (55) -- node[right]{} (65); 

	\draw[-=,line width=.5pt,double distance=2pt] (56) -- node[right]{} (66); 
	\draw[-=,line width=.5pt,double distance=2pt] (66) -- node[right]{} (76);

	\draw[-=,line width=.5pt,double distance=2pt] (57) -- node[right]{} (67); 
	\draw[-=,line width=.5pt,double distance=2pt] (67) -- node[right]{} (77);

	\draw[-=,line width=.5pt,double distance=2pt] (48) -- node[right]{} (58); 
	\draw[-=,line width=.5pt,double distance=2pt] (58) -- node[right]{} (68); 
	\draw[-=,line width=.5pt,double distance=2pt] (68) -- node[right]{} (78); 
	
	\draw[-{Implies},double distance=3pt,shorten >=2pt,shorten <=2pt,font=\scriptsize,] (4h) to node[scale=1] [right] {$\tau_h({u})$} (3h);
	\draw[-{Implies},double distance=3pt,shorten >=2pt,shorten <=2pt,font=\scriptsize,] (3g) to node[scale=1] [right] {$\tau_{g}({u})$} (2g);
	\draw[-{Implies},double distance=3pt,shorten >=2pt,shorten <=2pt,font=\scriptsize,] (3f) to node[scale=1] [right] {$\tau_{f}({u})$} (2f);
\end{tikzpicture}\]
and $\tau_{(h\circ g)\circ f}({u})$ is obtained from the following diagram 
\[\begin{tikzpicture}[scale=.9]
	\node (71) at  (0,6) {$\cA$ };
	\node (72) at (2,6) {$\cA$};
	\node (73) at (4,6) {$\cB$};
	\node (74) at (6,6) {$\cB$};
	\node (75) at (8,6) {$\cC$};
	\node (76) at (10,6) {$\cC$};
	\node (77) at (12,6) {$\cD$};
	\node (78) at (14,6) {$\cD$};

	\node (61) at  (0,4.5) {$\cA$ };
	\node (62) at (2,4.5) {$\cA$};
	\node (63) at (4,4.5) {$\cB$};
	\node (64) at (6,4.5) {$\cB$};
	\node (65) at (8,4.5) {$\cC$};
	\node (66) at (10,4.5) {$\cC$};
	\node (67) at (12,4.5) {$\cD$};
	\node (68) at (14,4.5) {$\cD$};
	
	\node (51) at  (0,3) {$\cA$ };
	\node (52) at (2,3) {$\cA$};
	\node (53) at (4,3) {$\cB$};
	\node (54) at (6,3) {$\cB$};
	\node (55) at (8,3) {$\cC$};
	\node (56) at (10,3) {$\cC$};
	\node (57) at (12,3) {$\cD$};
	\node (58) at (14,3) {$\cD$};

	\node (41) at  (0,1.5) {$\cA$ };
	\node (42) at (2,1.5) {$\cA$};
	\node (43) at (4,1.5) {$\cB$};
	\node (44) at (6,1.5) {$\cB$};
	\node (45) at (8,1.5) {$\cC$};
	\node (46) at (10,1.5) {$\cC$};
	\node (47) at (12,1.5) {$\cD$};
	\node (48) at (14,1.5) {$\cD$};

	\path[->,font=\scriptsize,>=angle 90]

	(71) edge node[above]  {${\alpha({u},{\unit}^{\op})}$} (72)
	(72) edge node[above]  {$ f $} (73)
	(73) edge node[above]  {${\beta({u},{\unit}^{\op})}$} (74)
	(74) edge node[below] (3g) {} node[above]  {$ g $} (75)
	(75) edge node[above]  {${\gamma({\unit},{u}^{\op})}$} (76)
	(76) edge node[above]  {$ h $} (77)
	(77) edge node[above]  {${\delta({\unit},{u}^{\op})}$} (78)

	(61) edge node[above]  {${\alpha({u},{\unit}^{\op})}$} (62)
	(62) edge node[above]  {$ f $} (63)
	(63) edge node[above]  {${\beta({\unit},{u}^{\op})}$} (64)
	(64) edge   node[above] (2g) {$ g $} (65)
	(65) edge node[above]  {${\gamma({u},{\unit}^{\op})}$} (66)
	(66) edge node[below] (4h) {} node[above] {$ h $} (67)
	(67) edge node[above]  {${\delta({\unit},{u}^{\op})}$} (68)

	(51) edge node[above]  {${\alpha({u},{\unit}^{\op})}$} (52)
	(52) edge node[below] (3f) {} node[above]  {$ f $} (53)
	(53) edge node[above]  {${\beta({\unit},{u}^{\op})}$} (54)
	(54) edge node[above] {$ g $} (55)
	(55) edge node[above]  {${\gamma({\unit},{u}^{\op})}$} (56)
	(56) edge node[above]  (3h) {$ h $} (57)
	(57) edge node[above]  {${\delta({u},{\unit}^{\op})}$} (58)

	(41) edge node[above]  {${\alpha({\unit},{u}^{\op})}$} (42)
	(42) edge node[above] (2f) {$ f $} (43)
	(43) edge node[above]  {${\beta({u},{\unit}^{\op})}$} (44)
	(44) edge node[above]  {$ g $} (45)
	(45) edge node[above]  {${\gamma({\unit},{u}^{\op})}$} (46)
	(46) edge node[above]  {$ h $} (47)
	(47) edge node[above]  {${\delta({u},{\unit}^{\op})}$} (48)
	
	(72) edge[bend left=30] node [above] {$(h\circ g)\circ f$} (77) 

(42) edge[bend right=30] node [below] {$(h\circ g)\circ f$} (47) 
	;

	\draw[-=,line width=.5pt,double distance=2pt] (41) -- node[right]{} (51); 
	\draw[-=,line width=.5pt,double distance=2pt] (51) -- node[right]{} (61); 
	\draw[-=,line width=.5pt,double distance=2pt] (61) -- node[right]{} (71);

	\draw[-=,line width=.5pt,double distance=2pt] (52) -- node[right]{} (62); 
	\draw[-=,line width=.5pt,double distance=2pt] (62) -- node[right]{} (72);

	\draw[-=,line width=.5pt,double distance=2pt] (53) -- node[right]{} (63); 
	\draw[-=,line width=.5pt,double distance=2pt] (63) -- node[right]{} (73);

	\draw[-=,line width=.5pt,double distance=2pt] (44) -- node[right]{} (54); 
	\draw[-=,line width=.5pt,double distance=2pt] (54) -- node[right]{} (64); 

	\draw[-=,line width=.5pt,double distance=2pt] (45) -- node[right]{} (55); 
	\draw[-=,line width=.5pt,double distance=2pt] (55) -- node[right]{} (65); 

	\draw[-=,line width=.5pt,double distance=2pt] (46) -- node[right]{} (56); 
	\draw[-=,line width=.5pt,double distance=2pt] (66) -- node[right]{} (76);

	\draw[-=,line width=.5pt,double distance=2pt] (47) -- node[right]{} (57); 
	\draw[-=,line width=.5pt,double distance=2pt] (67) -- node[right]{} (77);

	\draw[-=,line width=.5pt,double distance=2pt] (48) -- node[right]{} (58); 
	\draw[-=,line width=.5pt,double distance=2pt] (58) -- node[right]{} (68); 
	\draw[-=,line width=.5pt,double distance=2pt] (68) -- node[right]{} (78); 
	
	\draw[-{Implies},double distance=3pt,shorten >=2pt,shorten <=2pt,font=\scriptsize,] (4h) to node[scale=1] [right] {$\tau_h({u})$} (3h);
	\draw[-{Implies},double distance=3pt,shorten >=2pt,shorten <=2pt,font=\scriptsize,] (3g) to node[scale=1] [right] {$\tau_{g}({u})$} (2g);
	\draw[-{Implies},double distance=3pt,shorten >=2pt,shorten <=2pt,font=\scriptsize,] (3f) to node[scale=1] [right] {$\tau_{f}({u})$} (2f);
\end{tikzpicture}\]
Using the exchange law, since every unlabeled square is identity, we obtain that $$\tau_{(h\circ g)\circ f}({u})=\tau_{h\circ (g\circ f)}({u}).$$  
The remaining cases are similar to one of the cases above. It is also straightforward to check that the  diagram in \ref{diag:triangleofend} commutes. Therefore, we obtain that the composition in $\ii \frM$ is strictly associative. It is now easy to see that the unitors in $\ii \frM$ are also identity, which makes $\ii \frM$ a $\VCat$-category. 

 The category $\ii\frM$ does not have products, however it contains two nice subcategories; namely, the subcategories of objects with left $\ii$-actions and of objects with  right $\ii$-actions. Any shape of weak (co)limit that exists in $\frM$ also exists in both of these subcategories. 

\subsection{Stable {$\ii$}-objects}
\label{ssec:stable0cell}
We define the notion of stability for $\ii$-objects under an $\ii$-action. 
\begin{definition}\label{def:stable0cell}
	We say  a object $\cA_\alpha$ in  $\ii\frM$  is stable if $\alpha$ is mute on the right (resp. left) and there exits a $\ii $-action $\beta$ on a object $\cB$ that is mute on the left (resp. right) such that the objects of ${\cA}_{\alpha}$ and ${\cB}_{\beta}$  are  $1$-equivalent in  $\ii \frM$. The full sub-$\VCat$-category of stable  $\ii$--objects is  denoted by $\stVM$.  
\end{definition}

\begin{example}If every object in a monoidal category is tensor invertible, then any object  on which this monoidal category acts is stable. In particular, if we consider a group as a monoidal category, any object with an  action of this group is stable. 	
 \end{example}
\begin{example} If $\cA$ is a object in $\frM$, then it is stable under the $\bbN$-action generated by a self $1$-isomorphism of $\cA$. More generally, a object is stable under the action generated by a collection of its self $1$-isomorphisms.
 	
 \end{example}
\begin{example} The category of chain complexes is stable under the $\bbN$-action defined by shifts.  
 \end{example}
\begin{example}  Any stable derivator (see \cite[Defn. 5.5]{coley2019derivators}) is stable under the $\bbN$-action defined by suspensions. For more details on derivators, see \cite{heller,groth2013derivators}. \end{example}
 
The main motivation of the definition above follows from the notion of  spectra.  Historically, spectra is introduced due to the desire of having ``negative dimensional spheres" and the category of spectra. This is achieved by inverting the suspension or loop functors.  The category of spectra is   the result of   forcing the $\bbN$-action on the category of spaces given by suspensions (or loop space functors) to be reversed, so that it is a self-equivalence on this new category, see also \ref{sssec:sequentialspectra}. Since suspension functor has many other special properties, there are other equivalent ways of defining stable homotopy categories; however, the generalization above based on the most essential and  historical motivations; that is, reversing an action. Moreover, this definition also includes various other examples of spectra with different indexes.
\section{Stabilizations and costabilizations with respect to {$\ii$}-actions}\label{sec:stabilizationcostabilization}
 For this section we will define approximations to $\ii$-objects by stable $\ii$-objects. This approximations will be defined as  right and left biadjoints to the inclusion of stable $\ii$-objects into $\ii$-objects. For this section $\bfV$ denotes a sub-$\VCat$-category of $\VCat$ with objects $\{\ii,\unit\}$, with $1$-morpshisms containing the set of all functors  from $\ii$ to $\ii$ induced by the left and right multiplication in $\ii$ and the inlcusion functor $\unit\to \ii$.  
\subsection{Stabilization of {$\ii$}-objects}
 We define the notion of the stabilization as follows: 


\begin{definition}\label{def:stabilization}
	Let $\cA_{\alpha}$ be a object in $\ii\frM$.	 The stabilization of $\cA_{\alpha}$ is a stable object $\Stab_{\ii }(\cA)_{\underline{\alpha}}$, together with an $\ii$-equivariant morphism $\epsilon :\Stab_{\ii }(\cA)_{\underline{\alpha}}\to \cA_{\alpha}$ such that for every stable object $\cB_{\beta }$, the induced $\VV$-functor $\epsilon_*:\Fun_{\ii}(\cB_{\beta },\Stab_{\ii }(\cA)_{\underline{\alpha}})\to \Fun_{\ii }(\cB_{\beta },\cA_{\alpha})$ is an equivalence of $\VV$-categories. 
\end{definition}
Notice that the definition above implies for every $\ii$-equivariant morphism $f:\cB_{\beta}\rightarrow \cA_{\alpha}$ there exists an $\ii$-equivariant morphism $\widetilde{f}:\Stab_{\ii }(\cA)_{\underline{\alpha}}\rightarrow \cB_{\beta }$ together with a $2$-isomorphism 
	\begin{equation}\begin{tikzpicture}[scale=0.75]
			\node (s) at (2,0) {$\Stab_{\ii }(\cA)_{\underline{\alpha}}$};
			\node (c) at (0,2) {$\cB_{\beta}$};
			\node (a) at (4,2) {$\cA_{\alpha}$};
			\path[->,font=\scriptsize,>=angle 90]
			(s) edge node[below]  {$\epsilon$} (a)
			(c) edge node[below]  {$\widetilde{f}$} (s)
			(c) edge node[below] (f) {${f}$} (a);
			\draw[-{Implies},double distance=2pt,shorten >=2pt,shorten <=2pt] (s) to node[scale=1] [right] {} (f);
	\end{tikzpicture}\end{equation}
	where $\widetilde{f}$ is unique up to unique $1$-isomorphism.

 Suppose that $\frM$ is weakly powered over $\bfV$ and let the $\VCat$-functor  $\pitchfork:\bfV^{\op}\times \frM\to \frM$ denote the powering.  We  define two $\VV$-functors $\mu _r,\mu _l:\HH\times \HH^{\op}\to \bfV^{\op}$ as follows: For $s$ in $\{l,r\}$ define $\mu_s ((*,*),(*,*))=\ii $ considered as a category and let $\mu_s((x,y^{\op}),(z,w^{\op})^{\op}):\ii\rightarrow \ii $ be the $\VV$-functor given by
\begin{equation}\label{eqn:mu}
\mu_s ((x,y^{\op}),(z,w^{\op})^{\op})({u})=
\left\{
\begin{array}{ll}
  (y{\circledast} {u}){\circledast} w & \text{, if }s=r \\
  z{\circledast} ({u}{\circledast}  x) & \text{, if }s=l 
\end{array}
\right.\end{equation}
Let $\alpha:\HH\to \frM$ be an $\ii$-action on a object $\cA$ in $\frM$. Define $\Theta_{s,\alpha}$ as the composition 
\begin{equation}\label{eqn:theta}
\HH \times \HH \stackrel{id\times \psi}\rightarrow
\HH \times {\HH}^{\op}\times {\HH} \stackrel{\mu_s\times\alpha}\longrightarrow
\bfV^{\op}\times \frM\stackrel{\pitchfork}\longrightarrow \frM .
\end{equation}
Define an $\ii$-action $\Inv_{s}( {\alpha}):\HH\rightarrow\frM$ as
\begin{equation}\label{eqn:inv}
\Inv_{s}({\alpha})(-)={\oint}_{x:\mathbf{B}\ii^{\op}} \Theta_{s,\alpha}(-,(x,x))
\end{equation}
We use the notation $\Inv_{s}(\cA_{\alpha})$ for $\Inv_{s}({\alpha})(*,*)$.

 Let $\ddot{\mu}_l=\mu_{l}(({\unit},{\unit}^{\op}),-)$ and $\ddot{\mu}_r=\mu_{l}(-,({\unit},{\unit}^{\op})^{\op})$. Let $s,s' \in \{l,r\}$ with $s\neq s'$.
Note that $\Theta_{s,\alpha}(*,*,-)$ is equal to the composition given by
\[\HH  \stackrel{\psi}\rightarrow
{\HH}^{\op}\times {\HH} \stackrel{\ddot{\mu}_{s'}\times  \alpha}\longrightarrow
\bfV^{\op}\times \frM\stackrel{\pitchfork}\longrightarrow \frM .\]
which we denote by $\pitchfork(\ddot{\mu}_{s'},\alpha)$. Thus,  $\Inv_{s}( \cA_{\alpha})$ is isomorphic to 
\begin{equation}\label{eqn:inv(A)} {\oint}_{x:\mathbf{B}\ii^{\op}}\pitchfork(\ddot{\mu}_{s'},\alpha)({x,x})
\end{equation}
Observe that if $\frM=\CAT$, then $\Inv_{s}( \cA_{\alpha})\cong \Fun_{\ii}(\ii_{\ddot{\mu}_{s'}},\cA_{\alpha}).$

The inclusion ${\unit}\rightarrow \ii$ induces an $\ii$-equivariant morphism 
\[\epsilon_s:\Inv_{s}({\cA}_{\alpha})\rightarrow {\cA}_{\alpha}\] called the \emph{evaluation at ${\unit}$}, given by the composition of  $\omega_s:\Inv_{s}({\cA}_{\alpha})\to\pitchfork(\ii,\cA)$, the universal wedge of the weak end, and $\varepsilon:\ \pitchfork(\ii,\cA)\to\ \pitchfork({\unit},\cA)= \cA $, the usual evaluation.  We can define $$\sigma:[\Inv_{s}( \cA_{\alpha}), \alpha]({\unit},u^{\op})(\epsilon_s)\to [\Inv_{s}(\cA_{\alpha}), \alpha](u,{\unit}^{\op})(\epsilon_s)$$
as follows: Let $\tilde{\mu}_l=\mu_{l}(-,({\unit},{\unit}^{\op}))$ and $\tilde{\mu}_r=\mu_{l}(({\unit},{\unit}^{\op})^{\op},-)$. For every $u$ in $\ii$, we have a map
\[ \varepsilon \circ \pitchfork(\ddot{\mu}_{s'},\alpha)({\unit},{u}^{\op})\circ\omega_s \rightarrow  \varepsilon \circ \pitchfork(\ddot{\mu}_{s'},\alpha)({u},{\unit}^{\op}) \circ\omega_s \]
By naturality of $\varepsilon$, we obtain
\[ \alpha({\unit},{u}^{\op})\circ \varepsilon \circ \pitchfork(\ddot{\mu}_{s'},\id_A)({\unit},{u}^{\op})\circ\omega_s \rightarrow \alpha({u},{\unit}^{\op})\circ  \varepsilon \circ \pitchfork(\ddot{\mu}_{s'},\id_A)({u},{\unit}^{\op}) \circ\omega_s \]
Notice that 
$\varepsilon \circ \pitchfork(\tilde{\mu}_{s'},\id_A)({u},{\unit}^{\op})=\varepsilon \circ \pitchfork(\ddot{\mu}_{s'},\id_A)({\unit},{u}^{\op})$ and   $\varepsilon\circ\pitchfork(\tilde{\mu}_{s'},\id_A)({\unit},{u}^{\op})=\varepsilon \circ \pitchfork(\ddot{\mu}_{s'},\id_A)({u},{\unit}^{\op})$. 
Then  we obtain
\[ \alpha({\unit},{u}^{\op})\circ \varepsilon \circ \pitchfork(\tilde{\mu}_{s'},\id_A)({u},{\unit}^{\op})\circ\omega_s \rightarrow \alpha({u},{\unit}^{\op})\circ  \varepsilon\circ\pitchfork(\tilde{\mu}_{s'},\id_A)({\unit},{u}^{\op}) \circ\omega_s \]
By Fubini theorem, we get 
 \[ \alpha({\unit},{u}^{\op})\circ \varepsilon\circ\omega_s \circ \Inv_{s}({\alpha})({u},{\unit}^{\op}) \rightarrow \alpha({u},{\unit}^{\op})\circ  \varepsilon\circ\omega_s \circ \Inv_{s}( {\alpha})({\unit},{u}^{\op}) \]
We define $\sigma$ as 
\[\sigma({u}):\alpha({\unit},{u}^{\op})\circ \epsilon_s \circ \Inv_{s}({\alpha})({u},{\unit}^{\op}) \rightarrow \alpha({u},{\unit}^{\op})\circ  \epsilon_s \circ \Inv_{s}( {\alpha})({\unit},{u}^{\op}).\]

 For the rest of the subsection we assume $\Inv_{s}(\alpha)$ exists for every $\ii$-action $\alpha$. 
\begin{theorem}\label{thm:stableiffepsilonis1equivalence}
Let $\alpha $ be an right (respectively left) $\ii $-action on a object $\cA$ in a $\VCat$-category $\frM$. Assume $\Inv_{s}(\alpha)$ is defined. Then $\cA_\alpha $ is stable  if and only if $\epsilon_l$ (respectively $\epsilon_r$) is a  $1$-equivalence in $\ii\frM$.
\end{theorem}
\begin{proof}
The ``if" part follows from the definition of being stable. Now for the ``only if" part, assume that  $\cA_\alpha $ is stable. Without loss of generality assume $\alpha $ is a right action. Then $\cA_\alpha $ is  $1$-equivalent to $\cB_\beta $ for some left $\ii$-action $\beta $. Since $\pitchfork$ is a powering, $\epsilon _l:\Inv_{l }(\cB_\beta)\to \cB_\beta $ is a  $1$-equivalence. Also notice that the  $1$-equivalence from  $\cA_\alpha $ to $\cB_\beta $ induces a  $1$-equivalence from   $\Inv_{l }(\cA_{\alpha})$ and  $\Inv_{l }(\cB_\beta)$. Hence by naturality of the evaluation,   $\epsilon _l:\Inv_{l }(\cA_{\alpha})\to \cA_\alpha $ is a  $1$-equivalence.
\end{proof}

Given any object $\cA_{\alpha}$  in $\ii\frM$ let   \[\Inv_{\ii}(\cA_{\alpha})=\Inv_{l}(\Inv_{r}(\cA_{\alpha}))\] and  
$\epsilon = \epsilon_{r} \circ  \epsilon_{l}.$ Let $\Inv^{\unit}_{\ii }(\cA_{\alpha})= \Inv_{l}(\cA_{\alpha})$ and $\Inv^{n }_{\ii }(\cA_{\alpha})= \Inv_{\ii }( \Inv^{n-1 }_{\ii }(\cA_{\alpha}))$ for every $n>1$.  Define $\Inv_{\ii}^{\infty}(\cA_{\alpha })$ as the weak limit of the diagram
\[\cdots\longrightarrow  \Inv^{3}_{\ii }(\cA_{\alpha}) \overset{\epsilon}\longrightarrow  \Inv^{2}_{\ii }(\cA_{\alpha}) \overset{\epsilon}{\longrightarrow} \Inv^{\unit}_{\ii}(\cA_{\alpha}),\] in the $\VCat$-category $[\HH,\frM]$.  Note that every object in the above diagram is mute on the right, so that the object is isomorphic to an object in  $[\mathbf{B}\ii,\frM] $ (which is a subcategory of $\ii \frM$) under the inclusion induced by the functor $\HH \to \mathbf{B}\ii:  ({u},{v}^{\op})\mapsto u$.

Let $\epsilon^\infty: \Inv^{\infty }_{\ii }(\cA_{\alpha})\to \cA_\alpha$ be the $\epsilon_l$ composed with the  compositions of the maps in the diagram.

\begin{theorem}\label{thm:invinftyisstab}
Assume that $\Inv^{\infty }_{\ii }(\cA_{\alpha})$ exists. Then $\Inv^{\infty }_{\ii }(\cA_{\alpha})$ together with $\epsilon^\infty: \Inv^{\infty }_{\ii }(\cA_{\alpha})\to \cA_\alpha$  is a stabilization of $\cA_{\alpha }$.
\end{theorem}
\begin{proof}
 Since $\Inv_r$ is given by an weak end of a weak powering, it commutes with weak limits. Therefore,   $\Inv_{r}(\Inv^{\infty }_{\ii }(\cA_{\alpha}))$ is isomorphic to the weak limit $\lim_{n}\Inv_{r}(\Inv^{n}_{\ii}(\cA_{\alpha}))$. For each $n\geq 1$, we have  morphisms $\epsilon_{r}:\Inv_{r}(\Inv^{n}_{\ii}(\cA_{\alpha}))\to \Inv^{n}_{\ii}(\cA_{\alpha})$, so that there is a universal  morphism  $$F:\Inv_{r}(\Inv^{\infty }_{\ii }(\cA_{\alpha})) \to \Inv^{\infty }_{\ii }(\cA_{\alpha})$$ and we have  morphisms $\epsilon_l:\Inv^{n+1}_{\ii}(\cA_{\alpha})\to \Inv_{r}(\Inv^{n}_{\ii}(\cA_{\alpha}))$, so that there is a universal  morphism  $$G:\Inv^{\infty }_{\ii }(\cA_{\alpha}) \to \Inv_{r}(\Inv^{\infty }_{\ii }(\cA_{\alpha})).$$
 Hence, $\Inv^{\infty }_{\ii }(\cA_{\alpha})$ is stable. 

Given any stable $\cB_\beta$ in $\ii\frM$ we have an $\ii$-equivariant $1$-isomorphism  $\cB_\beta\to \Inv_{\ii}(\cB_{\alpha})$, so that $\Inv_{\ii}(\cB_{\alpha})$ is a also stable. Inductively, we obtain $\ii$-equivariant $1$-isomorphisms $\cB_\beta\to \Inv^{n}_{\ii}(\cB_{\alpha})$ for all $n\geq 1$. Thus,  there is a $1$-isomorphisms $\cB_\beta\to \Inv^{\infty}_{\ii}(\cB_{\alpha})$. Then, for any given $\ii$-equivariant morphism $f:B\to A$, there exist a unique $\tilde{f}:B \to \Inv^{\infty}_{\ii}(\cA_{\alpha})$ given by the composition 
$$\tilde{f}:B \overset{f}\lto \Inv^{\infty}_{\ii}(\cB_{\alpha})\overset{\Inv^{\infty}_{\ii}(f)}\lto \Inv^{\infty}_{\ii}(\cA_{\alpha})$$ such that $\epsilon^\infty\circ \tilde{f}$ is naturally isomorphic to $f$.  This implies that $$\epsilon^\infty_*:\Fun_{\ii }(\cB_{\beta },\Inv^{\infty}_{\ii}(\cA)_{\underline{\alpha}})\to \Fun_{\ii }(\cB_{\beta },\cA_{\alpha})$$ is an equivalence of categories. 
\end{proof}

In the case when $\ii $ is a symmetric monoidal category, we can consider $\Inv_{\ii }$ as the stabilization due to the following theorem.
\begin{theorem}\label{thm:vvissymmetricimpliesinvisstab}
	Assume that $\ii $ is a symmetric monoidal category. Then $\Inv_{\ii }(\cA_{\alpha})$ is equivalent to $\Inv^{\infty }_{\ii }(\cA_{\alpha})$.
\end{theorem}
\begin{proof} 
 Let $E:\Inv_{\ii }(\cA_{\alpha})\to \Inv_{l } (\Inv_{\ii }(\cA_{\alpha}))$ be the functor induced by 
$$\pi_2\times \mu\circ(\pi_1\times\pi_3):\ii\times \ii\times \ii\to \ii\times \ii:(u,v,w)\mapsto (v,u{\circledast} w).$$  It is straightforward to check that $E$ with $\epsilon_r:\Inv_{r}(\Inv_{\ii }(\cA_{\alpha}))\to \Inv_{\ii }(\cA_{\alpha})$ is a $1$-equivalence. 
\end{proof}

If $\frM=\CAT$, then the objects in this limit can be viewed as follows: 
The objects of $\Inv_{\ii}^{\infty}(\cA_{\alpha })$ are sequences of pairs  $\{(f_n  , {\phi_n} )\}_{n\in \bbN}$ where $f_n\in \Inv^{n}_{\ii }(\cA_{\alpha})$  and $\phi_{n}:f_{n+1}({\unit})({\unit})\to f_{n}$ is an isomorphism in $\Inv^{n}_{\ii }(\cA_{\alpha})$. A morphism $\zeta$ between two objects $\{({f_n},{\phi_{n}})\}$ and $\{({g_n},\varphi_{n})\}$ is a set of arrows ${\zeta_n:f_n\to g_n}$ satisfying the evident compatibility conditions. The action on $\Inv_{\ii}^{\infty}(\cA_{\alpha })$ is defined pointwise; that is, $\alpha^\infty({u},{v}^{\op})\{({f_n},{\phi_{n}})\}=\{(\widetilde{\alpha}^n_r({u},{v}^{\op})({f_n}),\widetilde{\alpha}^n_r({u},{v}^{\op})({\phi_{n}}))\}$ where $\widetilde{\alpha}^n_r$ is the action on $\Inv^{n}_{\ii }(\cA_{\alpha})$.

We can also see that	$\Inv^{\infty }_{\ii }(\cA_{\alpha})$ is stable under  the action  by using the object-wise description given above. 
	Define $$\psi:\Inv^{\infty }_{\ii }(\cA_{\alpha})\to\Inv_{r}(\Inv^{\infty }_{\ii }(\cA_{\alpha}))$$ by  
 $$\psi(\{(f_{n}  , {\phi_{n}} )\})({u})=\left\{(f_{n+1}({\unit})({u})  , {\phi_{n+1}}({\unit})({u}))\right\}$$ for any  $\{(f_n  , {\phi_n} )\}$  in $\Inv^{\infty }_{\ii }(\cA_{\alpha})$. 
 Then  $$(\epsilon_l\circ \psi)(\{(f_n  , {\phi_n} )\})=\psi(\{(f_n  , {\phi_n} )\})({\unit})=\left\{(f_{n+1} ({\unit})({\unit})  ,{\phi_{n+1} }({\unit})({\unit}))\right\}\cong \{(f_n  , {\phi_n} )\}$$
 and for any $g:\ii\to \Inv^{\infty }_{\ii }(\cA_{\alpha})$ in $\Inv_{r}(\Inv^{\infty }_{\ii }(\cA_{\alpha}))$ such that $g({u})=\{(f^{g({u})}_n,\phi^{g({u})}_n)\}$. Note that  $$g({u})\cong \widetilde{\alpha}^n_r({u},{v}^{\op})(g({\unit}))= \widetilde{\alpha}^n_r({u},{v}^{\op})(\{(f^{g({\unit})}_n,\phi^{g({\unit})}_n)\})$$ $$\cong \widetilde{\alpha}^n_r({u},{v}^{\op})(\{(f^{g({\unit})}_{n+1}({\unit})({\unit}),\phi^{g({\unit})}_{n+1}({\unit})({\unit}))\})\cong \{(\widetilde{\alpha}^n_r({u},{v}^{\op})(f^{g({\unit})}_{n+1}({\unit})({u})),\phi^{g({\unit})}_{n+1}({\unit})({u}))\}$$
Therefore
	\begin{align*}(\psi\circ \epsilon_l)({g})({u})=(\psi(g({\unit}))({u})&=\psi(\{(f^{g({\unit})}_n,\phi^{g({\unit})}_n)\})({u})
		\\
&=(f^{g({\unit})}_{n+1}({\unit})({u}),\phi^{g({\unit})}_{n+1}({\unit})({u}))\\
&\cong (f^{g({u})}_{n+1}({u})({u}),\phi^{g({u})}_{n+1}({u})({u}))\\
&\cong (f^{g({u})}_{n+1}({\unit})({\unit}),\phi^{g({u})}_{n+1}({\unit})({\unit}))\\
&= (f^{g({u})}_{n}, \phi^{g({u})}_{n})=g({u})\\
	\end{align*} 
Hence by induction $\phi(\epsilon_l(\widetilde{f}))=\widetilde{f}$.

In the case when $\ii $ is a symmetric monoidal category, the inverse equivalence $E:\Inv_{\ii }(\cA_{\alpha})\to \Inv_{l } (\Inv_{\ii }(\cA_{\alpha}))$  of  $\epsilon_r:\Inv_{r}(\Inv_{\ii }(\cA_{\alpha}))\to \Inv_{\ii }(\cA_{\alpha})$ in the proof of Theorem \ref{thm:vvissymmetricimpliesinvisstab} is given by $E(f)({u})({v})({w})=f({v})({u}{\circledast} {w})$. Observe that for every $g$ in $\Inv_{r}(\Inv_{\ii }(\cA_{\alpha}))$,
	\begin{align*}E(\epsilon(g))({u})({v})({w})  &= \epsilon(g)({v})({u}{\circledast} {w})	\\
	&=g({\unit})({v})({u}{\circledast} {w})\\
	&\cong g({u})({u}{\circledast} {v})({u}{\circledast} {w}) \\
	&\cong  g({u})({v}{\circledast} {u})({u}{\circledast} {w})  \\
	&\cong g({u})( {v})({w}). 
\end{align*}

\subsection{Costabilization of $\ii$-objects}
Dualizing the definition of the stabilization we obtain the costabilization.

\begin{definition}\label{def:costabilization}
Let $\cA_{\alpha}$ be a object in $\ii\frM$.		The costabilization of $\cA_{\alpha}$ is a stable object $\coStab_{\ii }(\cA)_{\underline{\alpha}}$, together with an $\ii$-equivariant morphism $\eta:\cA_{\alpha} \to \coStab_{\ii }(\cA)_{\overline{\alpha}}$ in $\ii\frM$ such that for every stable object $\cB_{\beta }$, the induced functor $\eta^*:\Fun_{\ii }(\cA_{\alpha},\cB_{\beta } )\to \Fun_{\ii }(\coStab_{\ii }(\cA)_{\overline{\alpha}},\cB_{\beta })$ is an equivalence of categories. 
\end{definition}
	Similar to the definition of stabilization, the definition above implies for every $\ii$-equivariant morphism $f:\cB_{\beta}\rightarrow \cA_{\alpha}$ in $\ii\frM$ such that $\beta $ is stable, there exists a morphism $\widetilde{f}:\cB_{\beta }\rightarrow \Stab_{\ii }(\cA)_{\underline{\alpha}}$ together with a $2$-isomorphism 
	\begin{equation}\begin{tikzpicture}[scale=0.75]
	\node (co) at (0,0) {$\coStab_{\ii }(\cA)_{\overline{\alpha}}$};
	\node (b) at (5,0) {$\cB_{\beta}$};
	\node (a) at (3,2) {$\cA_{\alpha}$};
	\path[->,font=\scriptsize,>=angle 90]
	(a) edge node[above]  {${f}$} (b)
	(co) edge node[above](f)  {$\widetilde{f}$} (b)
	(a) edge node[above]  {$\eta$} (co);
	\draw[-{Implies},double distance=2pt,shorten >=2pt,shorten <=2pt] (a) to node[scale=1] [right] {} (f);
	\end{tikzpicture}\end{equation}
	where $\widetilde{f}$ is unique up to unique $1$-isomorphism.

 Assume that $\frM$ is copowered over $\bfV$ and let the $\VCat$-functor $\odot:\bfV\times \frM\to \frM$ denote the copowering.  For $\cC$ in $\bfV$ and $\cA$ in $\frM$, we write  $\cC\odot \cA$ for $\odot(\cC,\cA)$. 
 Let $\varDelta:\HH\to \HH\times \HH$ be the diagonal functor. Let $\nu:\HH\times \HH\to \bfV$ be the functor given by 
 \begin{equation}\label{eqn:nu}
 	\nu_s ((x,y^{\op}),(z,w^{\op}))({u})=
 	\left\{
 	\begin{array}{ll}
 		(x{\circledast} {u}){\circledast} w & \text{, if }s=r \\
 		z{\circledast} ({u}{\circledast}  y) & \text{, if }s=l 
 	\end{array}
 	\right.\end{equation}
Given an $\ii$-action $\alpha:\HH\to \frM$ on a object $\cA$, define $\Upsilon^{s,\alpha}$ as the composition 
\begin{equation}\label{eqn:upsilon} 
	\HH \times \HH \stackrel{\id\times\varDelta }\rightarrow	\HH \times {\HH}\times {\HH} \stackrel{(\nu_s, \alpha)}\longrightarrow	\bfV\times \frM\stackrel{\odot}\longrightarrow \frM .\end{equation}
Define an $\ii$-action $\coInv_{s}( {\alpha}):\HH\rightarrow\frM$ as
\begin{equation}\label{eqn:coinv}\coInv_{s}({\alpha})(-)=\oint^{x:\mathbf{B}\ii^{\op}} \Upsilon^{s,\alpha}(-,(x,x))\end{equation}
We write ${\coInv_{r}(\cA_{\alpha})}$ for ${\coInv_{r}({\alpha})}(*,*)$.

Let $\ddot{\nu}_{l}=\nu_l (({\unit},{\unit}^{\op}),-)$ and $\ddot{\nu}_{r}=\nu_l (-,({\unit},{\unit}^{\op}))$. Then, for $s\neq s' $ in $\{l,r\}$, $\Upsilon^{s,\alpha}(*,*,-)$ is naturally isomorphic to the composition
\[\HH  \stackrel{\varDelta}\rightarrow
{\HH}\times {\HH} \stackrel{(\ddot{\nu}_{s'}\times \alpha)}\longrightarrow
\bfV\times \frM\stackrel{\odot}\longrightarrow \frM .\]
which we denote by $\ddot{\nu}_{s'}\odot\alpha$. Thus, $\coInv_{s}(\cA_{\alpha})$ is isomorphic to $$\oint^{x:\mathbf{B}\ii^{\op}}(\ddot{\nu}_{s'}\odot\alpha)({x,x})$$

 In this case,  the inclusion ${\unit}\rightarrow \ii$ induces an $\ii$-equivariant morphism 
\[\eta_s:{\cA}_{\alpha}\rightarrow\coInv_{s}({\cA}_{\alpha}) \]  which we call the \emph{coevaluation at ${\unit}$}. The $\ii$-equivarience of $\eta_s$ is similar to the case of $\epsilon_s$ above.

  For the rest of this subsection we  assume  $\coInv_{s}(\alpha)$ exists for every $\ii$-action $\alpha$. 

\begin{theorem}\label{thm:costableiffepsilonis1equivalence}
	Let $\alpha $ be an left (respectively right) $\ii $-action on a object $\cA$ in a $\VCat$-category $\frM$.  Then $\cA_\alpha $ is stable  if and only if $\eta_r$ (respectively $\eta_l$) is a  $1$-equivalence in $\ii\frM$.
\end{theorem}
\begin{proof}
This is dual to Theorem \ref{thm:stableiffepsilonis1equivalence}.
\end{proof}

Given any object $\cA_{\alpha}$  in $\ii\frM$ let   $\eta = \eta_l \circ \eta_{r}$. Let $\coInv^{\unit}_{\ii }(\cA_{\alpha})= \coInv_{l}(\cA_{\alpha})$ and $\coInv^{n }_{\ii }(\cA_{\alpha})= \coInv_{\ii }( \coInv^{n-1 }_{\ii }(\cA_{\alpha}))$ for every $n>1$.  Define $\coInv_{\ii}^{\infty}(\cA_{\alpha })$ as the colimit of sequence
\[\coInv^{\unit}_{\ii }(\cA_{\alpha})\overset{\eta }{\longrightarrow}\coInv^{2}_{\ii }(\cA_{\alpha}) \overset{ \eta }{\longrightarrow}\coInv^{3}_{\ii }(\cA_{\alpha}) \longrightarrow\cdots,\] provided that it exists.
Let $\eta^\infty: \cA_\alpha\to \coInv^{\infty }_{\ii }(\cA_{\alpha})$ be $\eta_l$ composed with the transfinite compositions of the maps in the diagram.
\begin{theorem}\label{thm:coinvinftyiscostab}
		Assume that $\coInv^{\infty }_{\ii }(\cA_{\alpha})$ exist. Then  $\coInv^{\infty }_{\ii }(\cA_{\alpha})$  with $\eta^\infty: \cA_\alpha\to \coInv^{\infty }_{\ii }(\cA_{\alpha}) $ is a costabilization of $\cA_{\alpha }$.
\end{theorem}
\begin{proof}
	The proof is dual to the proof of Theorem \ref{thm:invinftyisstab}.
\end{proof}

In the case when $\ii $ is a symmetric monoidal category, we can consider  $\co\Inv_{\ii }$ as a costabilization.
\begin{theorem}\label{thm:vvissymmetricimpliescoinviscostab}
If $\ii $ is a symmetric monoidal category, then $\coInv_{\ii }(\cA_{\alpha})$ is equivalent to $\coInv^{\infty }_{\ii }(\cA_{\alpha})$.
\end{theorem}
\begin{proof}
	The proof is dual to the proof of Theorem \ref{thm:vvissymmetricimpliesinvisstab}.
\end{proof}
Let $\ii $ be a symmetric monoidal category and $\frM=\CAT$. Note that  if ${\cA}_{\alpha}$ is mute on the right, then  $\coInv_{\ii }^{\infty }(\cA_{\alpha})$ is equivalent to   $\coInv_{r}(\cA_{\alpha})$. The category $\coInv_{r}(\cA_{\alpha})$ can be constructed (up to equivalence of categories) as follows. Let $\underline{\coInv_{r}(\cA_{\alpha})}$ be the category whose objects are the collections of all pairs $({u},a)$ and triples $({u},{v},a)$ where ${u},{v}$ are objects in $\ii$ and a an object in ${\cA}_{\alpha}$. A morphism between pairs is a morphism in  $ \ii\times \cA$ and a morphism triples is a morphism in $\ii\times \ii\times \cA$. A morphism from the triple $({u},{v},a_1)$ to a pair $({w},a_2)$ is either a morphism from 
$({u},a_1)$ to $({w},a_2)$ or a morphism from $({u}{\circledast} {v}, \alpha( {v},{\unit}^{\op})(a_1))$ to $({w},a_2)$.  There is no morphism from a pair to a triple in $\underline{\coInv_{r}(\cA_{\alpha})}$. Let $S$ be the collection of morphisms from triples to  pairs for  which are identity in both coordinates; that is, a morphism from $({u},{v},a_1)$ to  $({w},a_2)$ is in $S$ if either $a_1=a_2$ and  ${u}={w}$ and the corresponding map
$({u},a_1)$ to $({w},a_2)$ is identity or ${u}{\circledast} {v}={w}$ and  $ \alpha( {v},{\unit}^{\op})(a_1)=a_2$
 and the corresponding map from $({u}{\circledast} {v},\alpha( {v},{\unit}^{\op})(a_1))$ to $({w},a_2)$ is identity. Then  the category $\coInv_{r}(\cA_{\alpha})$ is given by the localization of $\underline{\coInv_{r}(\cA_{\alpha})}$ at $S$. Observe that every triple in $\underline{\coInv_{r}(\cA_{\alpha})}$  become isomorphic to a pair after passing to localization, ${\coInv_{r}(\cA_{\alpha})}$. Besides, every pair of the form $({u}{\circledast} {v},\alpha( {v},{\unit}^{\op})(a_1))$ becomes isomorphic to the pair $({u}, a_1)$. The $\ii$ action on $\coInv_{r}(\cA_{\alpha})$ is given by tensoring of first coordinate from the left. If $\cA$ is stable, then $\eta_r:\cA_{\alpha} \to {\coInv_{r}(\cA_{\alpha})}$ is an equivalence of categories. This can be realized via the isomorphism given by  $({u},a)\mapsto  \alpha({u},{\unit}^{\op})(a)$.

We say a $\VCat$-category $\frM$ is complete (resp. cocomplete) if it is weakly powered (resp. weakly copowered) over $\VCat$ and has all $\VCat$-enriched weak limits (resp. $\VCat$-enriched weak colimits). Combining Theorems \ref{thm:invinftyisstab} and \ref{thm:coinvinftyiscostab}, we obtain the following corollary.
\begin{corollary}\label{cor:main1}
	If $\frM$ is a complete (resp. cocomplete) $\VCat$-category, then $\stVM$ is   coreflective (resp. reflective) in $\ii\frM$.
\end{corollary} 
\begin{remark}
Note that if there is a monoidal category $\cG(\ii)$ in which every object is tensor invertible, and a functor $Q:\ii\to \cG(\ii)$ which is universal among functors from $\ii$ to monoidal categories with tensor invertible objects (i.e., $\cG(\ii)$ is a some sort of group completion of $\ii$) then $\stVM$ is equivalent to $[B\cG(\ii),\frM]$. In this case, we can directly stabilize one sided actions, and  the stabilization and costabilization coincides with the $\VCat$-categorical right and left Kan extensions along $BQ$.  
\end{remark}

\subsection{Lax stabilization and spectrification}\label{ssec:laxstabilizationandspectrification}
Let $\ii$ be symmetric monoidal and $\alpha:\HH\to  \frM $ be an $\ii$-action. Assume $\{s,s'\}=\{l,r\}$.  Define an $\ii$-action $\ell\Inv_{s}( {\alpha}):\HH\rightarrow\frM$ as
\begin{equation}\label{eqn:linv}
	\ell\Inv_{s}({\alpha})(-)=\sqint_{x:\mathbf{B}\ii^{\op}} \Theta_{s,\alpha}(-,(x,x))
\end{equation}
so that $\ell\Inv_{s}({\alpha})(*,*)=\ell \Inv_{s}( \cA_{\alpha})$ is isomorphic to 
\begin{equation}\label{eqn:linv(A)} \sqint_{x:\mathbf{B}\ii^{\op}}\pitchfork(\ddot{\mu}_{s'},\alpha)({x,x})
\end{equation}
Assume that $\mu$ is mute on the right (left) and $s=r$ ($s=l$). Then we can call $\ell\Inv_{s}(\cA_{\alpha})$ the lax stabilization of $\cA_\alpha$. There is a canonical morphism $$\iota_\alpha: \Inv_{s}(\cA_{\alpha})\to \ell \Inv_{s}(\cA_{\alpha}).$$ We call the left adjoint of $\iota_\alpha$ in $\ii\frM$ the \emph{spectrification} whenever it exists. 

Assume that $\alpha$ factors through the category of adjunctions $\Adj(\frM)$ whose morphisms are adjunctions in $\frM$. Then, there is an adjoint action $\bar{\alpha}$ on $\cA$ so that for each $u,v$ in $\ii$, $\alpha(u,v^{\op})$ is right adjoint to $\bar{\alpha}(v,u^{\op})$. Then there exist a equivariant morphism $$\hat{{\alpha}}:\cA_{\bar{\alpha}}\to\ell\Inv_{s}(\cA_{\alpha})$$ induced by the isomorphism $[\cA, \pitchfork(\ii,\cA)]\cong \CAT(\ii,[\cA,\cA]).$ Therefore, if there is a spectrification, then there exist an $\ii$-equivariant morphism $\cA_{\bar{\alpha}}\to\Inv_{s}(\cA_{\alpha})$. Composing with $\epsilon_{s}$ we obtain an $\ii$-equivariant morphism $\cA_{\bar{\alpha}}\to\cA_{\alpha}.$

In the case $\frM=\CAT$ and $s=r$, ${\hat{\alpha}}$ sends an object $a$ in $\cA$  to a pair $(\hat{\alpha}(a),\sigma_{\hat{\alpha}(a)})$ where $\hat{\alpha}(a)(u)={{\bar{\alpha}}}(\unit,u^{\op})(a)$ and $$\sigma_{\hat{\alpha}(a)}(u)_v:{{\bar{\alpha}}}(\unit,u^{\op})(a)\to \alpha(v,\unit^{\op})({{\bar{\alpha}}}(\unit,(u{\circledast} v)^{\op})(a))$$ is the adjunct of the identity map $\id_{{{\bar{\alpha}}}(\unit,u^{\op})(a)}:{{\bar{\alpha}}}(\unit,u^{\op})(a)\to {{\bar{\alpha}}}(\unit,u^{\op})(a)$ with respect to the adjunction $({{\bar{\alpha}}}(\unit,v^{\op})\dashv \alpha(v,\unit^{\op}))$. If $L$ denotes the spectrification, then we obtain an $\ii$-equivariant functor $\cA_{\bar{\alpha}}\to\cA_{\alpha}$ is given by the composition $$\cA_{\bar{\alpha}}\overset{\hat{\alpha}}\lto\ell\Inv_{r}(\cA_{\alpha})\overset{L}\lto \Inv_{r}(\cA_{\alpha})\overset{\epsilon_{r}}\lto \cA_{\alpha}.$$
In this case, lax stabilization of $\cA$ with respect to $\alpha$ and with respect to $\bar{\alpha}$ are isomorphic as categories. In fact, the category $\ell\Inv_{r}(\cA_{\bar{\alpha}})$ consists of pairs $(F,\varsigma)$ where $F:\ii\to \cA$ is a functor and $\varsigma_{u,v}:\bar{\alpha}(\unit,v^{\op})(E(u))\to E(u{\circledast} v)$ is a morphism that is natural in both $u$ and $v$. Given such a pair, we can define a new pair $(F,\sigma)$ where $\sigma_{u,v}:E(u)\to \alpha(v,\unit^{\op})( E(u{\circledast} v))$ is the adjunct of $\varsigma_{u,v}$. This pair defines an object in $\ell\Inv_{r}(\cA_{\alpha})$. Similarly, any object in $\ell\Inv_{r}(\cA_{\alpha})$ defines an object  in  $\ell\Inv_{r}(\cA_{\bar{\alpha}})$. Therefore, the association $(F,\varsigma) \mapsto (F,\sigma)$ define an isomorphism of categories; i.e., we have
$$\ell\Inv_{r}(\cA_{\alpha}) \cong \ell\Inv_{r}(\cA_{\bar{\alpha}}).$$

\section{Examples of stable $\ii$-objects and stabilizations}
\label{sec:examplesofstableactions}

Our primary examples of stable objects and stabilizations come from reduced (co)homology theories. Traditionally, (co)homology theories are measurement tools that compare homotopy properties of objects via comparing corresponding algebraic data. However, one does not have to use algebraic data to make such comparisons. It can be done by means of any functor between any two category. If $\cA$ and $\cB$ are two categories, and $F:\cA\to \cB$ is a functor between them than if $Fx\ncong Fy$ in $\cB$ then $x\ncong y$ in $\cA$. More generally, if $F$ is a functor that preserves certain structure (such as homotopy) than $Fx$ and $Fy$ not sharing this structure implies neither do $x$ and $y$. In this paper we propose a more general and unorthodox definitions of a homology and cohomology functors, whose source and target are homotopy categories, where the target has the notion of ``exact sequence". For the sake of computational tools, as in the classical (co)homology theories, we require these functors to satisfy certain axioms similar to the classical Eilenberg-Steenrod axioms. Notions of suspension and loop functors, which are in fact actions of certain monoidal categories, are also generalized  accordingly.

\subsection{$\ii$-graded (co)homology theories}
\label{ssec:vvgradedtheories}
For this section we let $\ii$ be symmetric monoidal and consider $\ii$-actions on categories that are $\pi_0$ of pointed $(\infty,1)$-categories; in particular, homotopy categories of certain pointed homotopy theories. Therefore, in the categories on which  $\ii$ acts, homotopy fiber and cofiber sequences exists. Let $f:\cA\to \cS$ be a functor between such categories where $\cS$ admits a notion of exact sequences (e.g., an Abelian category). We say $f$ is \emph{left exact} if it sends homotopy fiber sequences to exact sequences, and \emph{right exact} if it sends homotopy cofiber sequences to exact sequences.

We first consider a general definition of cohomology and homology functors in view of Eilenberg-Steenrod axioms. Let $\cA$ and $\cS$ be two pointed categories such that $\cS$ has exact sequences. 
	\begin{description}
		\item[Cohomology Functors]  A product preserving right exact functor $h:\cA^{\op}\to \cS$ is called a cohomology functor. The full subcategory of cohomology functors in $[\cA^{\op},\cS]$ is denoted by $\cohml(\cA,\cS)$.
		\item[Homology Functors] A coproduct preserving left exact functor $h:\cA\to \cS$ is called a homology functor. The full subcategory of homology functors in $[\cA,\cS]$  is denoted by $\hml(\cA,\cS)$.
	\end{description}

Let $\alpha$ be an $\ii$-action on $\cA^{\op}$ that is mute on the left. Consider $\cS$ with the trivial action. Then  the functor category $[\cA^{\op},\cS]$ admits an $\ii$-action as described in \ref{eqn:actiononAB}, which is mute on the right. We denote this action by $[\alpha,1]$. Under certain conditions  $[\alpha,1]$ restricts to the subcategory $\cohml(\cA,\cS)$ of cohomology functors. For example, it is enough to assume $\alpha$ is cocontinuous  (preserves $(\infty,1)$-colimits that exist in $\cA$).

We define category of cohomology theories graded over $\ii$ as follows:
\begin{definition}\label{def:vvgradedcohomologytheories} The category of cohomology theories graded over $\ii$ is defined as \[\COHML_\ii(\cA,\cS)=\Stab_{\ii }(\cohml(\cA,\cS)_{[\alpha , 1]});\] i.e., the stabilization of cohomology functors with respect to $[\alpha , 1]$. We  say that a cohomology theory graded over $\ii$ is an object in the category of cohomology theories.
\end{definition}

We can similarly define homology theories. Suppose now that    $\alpha$ is an $\ii$-action on $\cA$ that is mute on the left, so that the induced action $[\alpha,1]$ on  $[\cA,\cS]$ restricts to $\hml(\cA,\cS)$ (e.g.,  $\alpha$ is cocontinuous).
\begin{definition}\label{def:vvgradedhomologytheories}
	The category of homology theories graded over $\ii$ is defined as \[\HML_\ii(\cA,\cS)=\Stab_{\ii}(\hml(\cA,\cS)_{[\alpha,1]});\] i.e., the stabilization of homology functors with respect to action induced by $[\overline{\alpha }, 1]$.  We say that a homology theory graded over $\ii$ is an object in the category of homology theories.\end{definition}

\subsubsection{Axiomatic interpretation of the definitions above }\label{sssec:axiomsofvvgradedtheories}
Unfolding the definitions above we can see that a cohomology theory graded over $\ii$ consists of a functor $h:\ii \times\cA^{op}\to \cS$ such that each $h^u=h({\unit})({u})$ is a cohomology functor, together with natural isomorphisms $$\sigma_{u,v}: h^u \to  h^{{u}{\circledast} {v}}\circ \alpha({\unit},{v}^{\op})$$ such that  the following diagram commutes
\begin{equation}\begin{tikzpicture}[scale=1.55]
		\node (AB2) at (2,1) {$h^{u}(a)$};
		\node (BC) at (4,0) {$h^{{u}{\circledast} {{v}}{\circledast} {{w}} }(\alpha({\unit},({{v}}{\circledast} {{w}})^{op})(a))$};
		\node (AB) at (0,0) {$ h^{{u}{\circledast} {v}}(\alpha({\unit},{{v}}^{op})(a))$};

		\path[->,font=\scriptsize,>=angle 90]
		
		(AB) edge node[below]{$\sigma_{{u}{\circledast} {v},{w}} $} (BC)
		(AB2) edge node[right]{$\sigma_{{u},{{v}}{\circledast} {{w}}}$} (BC)
		(AB2) edge node[left]{$\sigma_{{u} ,{v}}$} (AB);

\end{tikzpicture}\end{equation}
for every $u,{v},{w}$  in $ \ii$. 
This diagram is induced by the commuting triangle \ref{diag:triangleofend}. Moreover, naturality of $\sigma$ implies that the following diagram commutes
\begin{equation}\begin{tikzpicture}[scale=1.55]
		\node (AB2) at (1,1) {$h^{u}(a)$};
		\node (BC2) at (4,1) {$ h^{{u}{\circledast} {v}}(\alpha({\unit},{{v}}^{op})(a))$};
		\node (AB) at (1,0) {$ h^{{u}{\circledast} {w}}(\alpha({\unit},{{w}}^{op})(a))$};
		\node (BC) at (4,0) {$h^{{u}{\circledast} {v}}(\alpha({\unit},{{w}}^{op})(a))$};
		\path[->,font=\scriptsize,>=angle 90]
		(AB) edge node[above]{$h^{{u}{\circledast} m}(\id) $} (BC)
		(AB2) edge node[above]{$\sigma_{u,v} $} (BC2)
		(AB2) edge node[right]{$\sigma_{u,w}$} (AB)
		(BC2) edge node[right]{$h^{{u}{\circledast} {v}}(\alpha({\unit},{m}^{op})(a))$} (BC);
\end{tikzpicture}\end{equation}
for  any morphism $m:{{v}}\to {{w}}$ in $\ii$.

Similarly, a homology theory graded over $\ii$ consists of a functor $h:\ii \times\cA\to \cS$ such that for every $u$ in $\ii$, the functor $h_u$ given by $h_u(a)=h({\unit},{u})(a)$ is a homology functor, together with natural isomorphisms $$\sigma_{u,v}:   \alpha({\unit},{v}^{op})\circ     h_{{u}{\circledast} {v}} \to h_u$$ satisfying similar conditions dual to the ones above.

\subsection{Some known examples of (co)homology theories}\label{ssec:knownexamplesofvvgradedtheories}
Definitions of several existing cohomology and homology theories fit into the setting that described above, after choosing actions appropriately. Some examples of actions below are non-strict; however, they are particular cases of the setting above in view of Remark \ref{rem:pseudotostrict}. In the examples below, a space means a compactly generated and weakly Hausdorff topological space. The first two of the examples below are also mentioned in first author's thesis \cite{erdal} as examples of (co)homology theories graded by monoidal categories. However, the approach therein  does not include the universal descriptions given here.
\subsubsection{Generalized ordinary (co)homology theories}\label{sssec:ordinarytheories}
Let $\mathbb{N}$ be the natural numbers considered as a monoidal category with identity morphisms as the only morphisms. Let $\cAb$ denote the category of abelian groups. Consider $\cAb$ with the trivial $\bbN$-action. Let $\ho\cT$ be the homotopy category of pointed spaces with respect to Quillen model structure. Define  the  $\bbN$-action $\Sigma$ by  \[\Sigma (n,m^{\op})(X)=X\wedge \mathbb{S}^m=\Sigma^mX\] for $n,m$ in $\mathbb{N}$ and $X$ in $\ho\cT$. Then a generalized (co)homology theory $h$ in our setting gives a generalized (co)homology theory satisfying the Eilenberg-Steenrod (co)homology axioms.

	\subsubsection{Equivariant cohomology theories graded over representations}
\label{sssec:ROGradedCohomologyTheories}
Let $G$ be a compact lie group and $U$ be a complete $G$-universe; that is, a countably infinite dimensional orthogonal $G$-representation having non-zero $G$-fixed points and contains the direct sum $V^{\oplus \lambda}$ for every finite dimensional representation $V$ and every cardinal $\lambda\leq \aleph_0$, see \cite[Defn. 1.1.12]{schwedeglobal} or \cite[Ch. II, Defn. 1.1]{mandellmay}.   Let $\ii $ be the monoidal category $\mathcal{RO}(G;U)$ whose objects are orthogonal $G$-representations embeddable in $U$ and whose morphisms are $G$-linear isometric isomorphisms and monoidal product is the direct sum. For $V$ an object in $\ii $, denote by $\bbS^V$ the one point compactification of $V$. Let $h\mathcal{RO}(G;U)$ be the quotient of $\mathcal{RO}(G;U)$ with the relation given by $f\sim g:V\to W$ if the induced maps $f_*\simeq_s g_*:\bbS^V\to \bbS^W$  are stably homotopic \cite[see pp.130]{may}. Let $G\cT$ denote the category of pointed $G$-spaces and pointed $G$-maps with the standard model structure and $\ho G\cT $ be its homotopy category. Let $\cAb$ be the category of abelian groups and $1$ denote the trivial $\ii $-action on $\cAb$.  We define an action $\Sigma$ on $\ho G\cT^{\op}$  as follows:
\[\Sigma (V,W^{\op})(X)= \Sigma^WX=X\wedge \bbS^W;\]
that is, the usual suspension action. For each $V,W$ in $\ii$ the suspension functor $\Sigma^W$  preserves homotopy colimits in $G\cT$, see \cite{lewis1986equivariant}. Hence $\Sigma$ induces an action on $\cohml(\ho G\cT,\cAb)$.  An $RO(G)$-graded cohomology theory is an object in the stabilization of $\cohml(\ho G\cT,\cAb)$ with respect to the action $[\Sigma,1]$. Therefore, a cohomology theory is a pair $(h,\sigma)$ where $h:h\mathcal{RO}(G;U)\times \ho G\cT^{\op}\to \cAb$ is a functor with $h^V=h(V,-)$ a cohomology functor for every $V$ in $h\mathcal{RO}(G;U)$, and $$\sigma_{V,W}: h^V \to  h^{{W}\oplus {V}}\circ \Sigma^W$$ are natural isomorphisms. For each pair of representations $W,Z$, as above, we have the following diagram commutes
	\[\begin{tikzpicture}[scale=1]
		\node (0) at (6,0) {$h^{V\oplus W\oplus Z}(\Sigma^{(W\oplus Z)}X  )$};
		\node (a) at (3,1.5) {$h^{V}(X)$};
		\node (b) at (0,0) {$h^{V\oplus W}(\Sigma^W X )$};
		\path[->,font=\scriptsize,>=angle 90]
		(a) edge node[left]  {$\sigma_{V,W}$} (b)
		(b) edge node[below] {$\sigma_{V\oplus W,Z}$} (0)
		(a) edge node[right] {$\sigma_{V,W\oplus Z}$} (0);
	\end{tikzpicture}.\]
If $m : W\rightarrow W'$ is a morphism in $h\mathcal{RO}(G;U)$, then by naturality of $\sigma$ we have the following diagram commutes
 \[\begin{tikzpicture}[scale=.8]
 	\node (0) at (1,0) {$h^{V\oplus W'}(  \Sigma^{W'}X)$};
 	\node (c) at (9,0) {$h^{V\oplus W'}( \Sigma^WX )$};
 	\node (a) at (1,2) {$h^{V}(X)$};
 	\node (b) at (9,2) {$h^{V\oplus W}(  \Sigma^WX)$};
 	\path[->,font=\scriptsize,>=angle 90]
 	(a) edge node[above]  {$\sigma_{V,W}$} (b)
 	(b) edge node[right] {$h(id\oplus m)$} (c)
 	(a) edge node[left] {$\sigma_{V,W'}$} (0)
 	(0) edge node[above] {$h^{V\oplus W'}(id\wedge \bbS^{m })$ } (c);
 \end{tikzpicture}\]
This definition of cohomology theory is the same as the one given in \cite[VII, Defn. 1.1]{may}.
Dualizing the definition, one obtains $RO(G)$-graded equivariant homology theories.

\subsubsection{Equivariant cohomology theories graded over actions on spheres}
\label{sssec:spheregradedCohomologyTheories}
Let $\ii$ be the category whose objects are pointed spheres (with $S^{-1}=*$) with continuous base-point preserving $G$-actions, morphisms are pointed $G$-isomorphisms between them. It is a monoidal category with $G$-smash product with diagonal action and $S^{-1}=*$ as the unit. We denote an object in $\ii$ by $S^m_\mu$, where $n$ is the dimension of the sphere and $\mu$ is the $G$-action.  Note that $\mathcal{RO}(G;U)$ faithfully embeds in $\ii$.  Define  $\Sigma$ on $\ho G\cT^{\op}$  similar to the above; i.e., 
\[\Sigma (S^n_\xi,(S^m_\mu)^{\op})(X)= \Sigma^{S^m_\mu}X=X\wedge S^m;\]
with the diagonal action. Clearly, $\Sigma$ induces an action on $\cohml(\ho G\cT,\cAb)$.  An $\ii$-graded cohomology theory is an object in the stabilization of $\cohml(\ho G\cT,\cAb)$ with respect to the action $[\Sigma,1]$. Therefore, an $\ii$-graded  cohomology theory is a pair $(h,\sigma)$ where $h:\ii\times \ho G\cT^{\op}\to \cAb$ is a functor with $h^{S^n_\mu}=h({S^n_\mu},-)$ a cohomology functor for every ${S^n_\mu}$ in $\ii$, and $$\sigma_{{S^m_\mu},{S^n_\xi}}: h^{S^n_\xi} \to  h^{{S^m_\mu}\wedge {{S^n_\xi}}}\circ \Sigma^{S^m_\mu}$$ are natural isomorphisms. The compatibility diagrams commute as above.

\subsubsection{Parameterized cohomology theories graded over vector bundles}
\label{sssec:bundlegradedparameterizedtheories}
For the original definitions of parameterized reduced cohomology theories, see \cite{mayparametrized}. These theories are cohomology theories for the categories of ex-spaces. Let $B$ be a space and $\cT/ B$ be the category of spaces over $B$. The terminal object in this category is the identity of $B$. The category of based spaces over $B$,  $\cT_B$, is the under category $\cT_B=\id_B/(\cT/B)$. Objects of this category are often called \emph{ex-spaces} of $B$, see e.g. \cite{james1995fiberwise}. Let $\ho\cT_B$ be the homotopy category of $\cT_B$ with respect to the model structure of  \cite[Ch. 6]{mayparametrized}. Ordinary parameterized cohomology theories takes values from  $\ho\cT_B$  and graded over integers in \cite{mayparametrized}. On the other hand, this category admits other obvious suspensions than the ordinary one. 

Let $\ii$ be the monoidal category whose objects are real vector bundles over $B$ with fiberwise inner products (i.e., taking values in $B\times R$ the trivial $\mathbb R$-bundle over $B$), fiberwise linear isometric isomorphisms as morphisms, Whitney sum $\oplus_B$ as the monoidal product and $0$-bundle as the monoidal unit.  For $\xi $ in $\ii$, denote by $\bbS^\xi$ the associated fiber-wise one point compactification, which is a sphere bundle over $B$ with point as the section induced by the zero section of $\xi$. Given bundles $\xi,\eta$ in $\ii$, define $\Sigma_B$ by \[\Sigma_B(\xi,\eta^{op})(t:X\to B)=t \wedge_B \bbS^{\eta},\] see \cite[Definition 1.3.3]{mayparametrized} for $\wedge_B $. The construction $\Sigma_B$  defines an action on $\ho\cT_B$. Moreover, for every $\xi,\eta$ in $\ii$ $\Sigma_B(\xi,\eta^{op})$ preserves homotopy colimits in $\cT_B$. Consider  $\cAb$   with the trivial $\ii$-action $1$. Then the $\ii$-action $[{\Sigma }_B, 1]$ on $[ \ho\cT_B^{\op},\cAb]$ induces a $\ii $-action on
$\cohml(\ho\cT_B,\cAb)$. A parameterized cohomology theory graded over $\ii$ is an object $\Stab_{\ii }(\cohml(\ho\cT_B,\cAb))$. Since $\ii$ is symmetric monoidal, the category $\Stab_{\ii }(\cohml(\ho\cT_B,\cAb))$ is equivalent to $\Inv_l(\cohml(\ho\cT_B,S))$. Then  a parameterized cohomology theory graded over $\ii$ is a pair $(h,\sigma)$ where $h:\ii \to \cohml(\ho\cT_B,\cAb)$ is a functor and $\sigma$ is the associated desuspension. More precisely, a parameterized cohomology theory graded over $\ii$ consists of a functor  $h:\ii\times \ho \cT_B^{\op}\to \cAb$ such that for every $\xi$ in $\ii$ $h^\xi =h({\unit})(\xi) $ is a cohomology functor in $\ho\cT_B$, together with natural isomorphisms 
\[\sigma_{\xi,\eta}:h^\xi \rightarrow h^{\xi \oplus_B \eta}\circ   \Sigma_B^\xi\]
for every  $\xi,\eta$ in $\ii$. Moreover, for every object  $\tau:X\to B$ in $\ho\cT_B$ the following diagram commutes
	\[\begin{tikzpicture}[scale=1]
	\node (0) at (6,0) {$h^{\xi\oplus_B \eta\oplus_B \zeta}(\Sigma_B^{(\eta\oplus_B \zeta)}\tau  )$};
	\node (a) at (3,1.5) {$h^{\xi}(\tau)$};
	\node (b) at (0,0) {$h^{\xi\oplus_B \eta}(\Sigma_B^\eta \tau )$};
	\path[->,font=\scriptsize,>=angle 90]
	(a) edge node[left]  {$\sigma_{\xi,\eta}$} (b)
	(b) edge node[below] {$\sigma_{\xi\oplus_B \eta,\zeta}$} (0)
	(a) edge node[right] {$\sigma_{\xi,\eta\oplus_B \zeta}$} (0);
\end{tikzpicture}.\]
  If $f : \eta\rightarrow \eta'$ is a morphism in $\ii$, then we have a commutative diagram as follows:
	\[\begin{tikzpicture}[scale=.8]
		\node (0) at (1,0) {$h^{\xi\oplus_B \eta'}(\Sigma_B^{\eta'} \tau)$};
		\node (c) at (9,0) {$h^{\xi\oplus_B \eta'}(\Sigma_B^\eta \tau)$};
		\node (a) at (1,2) {$h^{\xi}(\tau)$};
		\node (b) at (9,2) {$h^{\xi\oplus_B \eta}(\Sigma_B^\eta \tau)$};
		\path[->,font=\scriptsize,>=angle 90]
		(a) edge node[above]  {$\sigma_{\xi,\eta}$} (b)
		(b) edge node[right] {$h(id\oplus_B f)$} (c)
		(a) edge node[left] {$\sigma_{\xi,\eta'}$} (0)
		(0) edge node[above] {$h^{\xi\oplus_B \eta'}(id\wedge_B\bbS^{f })$ } (c);
	\end{tikzpicture}\]
	since $\sigma $ is natural transformation.
Passing to isomorphism classes in $\ii$, one obtains $KO_0(B)$-graded parameterized cohomology theories (for a slightly different definition see \cite[Sec.1]{prieto1987ko}).

One can also pass to bundles fibrewise embeddable in a universe bundle for convenience. Let $\mathfrak{u}$ be a vector bundle in $\ii$, such that for every finite dimensional vector bundle $\xi$ in $\ii$ and for every $\lambda\leq \aleph_0$,  there is a monomorphism $\xi^{{\circledast}_B \lambda}\to \mathfrak{u}$ of vector bundles. Let $\ii_{\mathfrak{u}}$ be the subcategory of $\ii$ whose objects are such $\xi$'s and whose morphisms are isomorphisms in $\ii$ between them. Denote by $\Sigma_{B^{\mathfrak{u}}}$ the restriction of the action $\Sigma_B$. Then, the associated stabilization of cohomology functors with respect to the action $\ii_{{\mathfrak{u}}}$-action $[\Sigma_{B^{\mathfrak{u}}}, 1]$ gives the parameterized cohomology theories indexed by a universe.

One can go further and grade parameterized cohomology theories over the symmetric monoidal category $\underline{\ii}$ of sphere bundles over $B$ admitting sections, with fiberwise isomorphisms of pointed bundle maps as morphisms. More precisely, objects of $\underline{\ii}$ are pairs $(\xi,s)$ where $\xi:E\to B$ is a sphere bundle and $s:B\to E$ is a section. The monoidal product is given by the fiber-wise smash product on the first coordinate and the unique section of the pushout in  \cite[Definition 1.3.4]{mayparametrized}  on the second coordinate. The action is defined similarly; that is, given objects $(\xi,s),(\eta,t)$ in $\underline{\ii}$, $\underline{\Sigma_B}$ is given by  \[\underline{\Sigma_B}((\xi,s),(\eta,t)^{op})(\tau:X\to B)=\tau \wedge_B \eta.\] Then  $\underline{\Sigma_B}$ induces an action on $\cohml(\ho\cT_B,\cAb)$ and an object in the stabilization of $\cohml(\ho\cT_B,\cAb)$ with respect to $\underline{\Sigma_B}$ defines a parameterized cohomology theory graded over sphere bundles admitting sections. The properties enjoyed by such a theory can be obtained similar to the one described above. 


{\section{(Co)stabilizations of relative categories with respect to $\ii$-actions}\label{sec:costabilizationsofrelativecategories}
The category $\RelCat$ of relative categories and relative functors between them  admits a model structure due to Barwick-Kan \cite[6.1]{barwick2012relative}. We assume every relative category is semi-saturated; i.e., every isomorphism is a weak equivalence. Besides, $\RelCat$  is a cartesian closed symmetric monoidal category \cite[7.1]{barwick2012relative}. It is in particular, a $2$-category ($\Cat$-enriched) and powered and copowered over any small category.
If  $I$ is a small category and $\cA=(A,W_A)$ is a relative category, then the powering $\pitchfork(I,\cA)$ is given by $(A^I,W^I_A)$, where $W^I_A$ denotes the pointwise weak equivalences. Similarly, the copowering $I \odot \cA$ is given by  $(I\times A, I\times W_A)$. In this section we assume  $\ii$ is symmetric monoidal.

\subsection{Stabilization of relative categories with respect to $\ii$-actions} We here first discuss stabilizations in the $2$-category of relative categories, relative functors and natural transformations, without referring the homotopy theory of $\RelCat$. Let $\alpha$ be an $\ii$-action on $\cA=(A,W_A)$ that is mute on the right (i.e., a left action). Then, by Theorem \ref{thm:vvissymmetricimpliesinvisstab},  $\Stab_{\ii}(\cA_\alpha)$ is equivalent to $\Inv(\cA_{\alpha})$. If $\alpha$ is mute on the right, then we have also an equivalence  $\Inv(\cA_{\alpha})\cong  \Inv_l(\cA_{\alpha})$. The objects of $\Inv_l(\cA_{\alpha})$ are pairs $(E,\sigma)$ where $E:\ii\to A:u\mapsto E_u$ is a relative functor (with trivial relative structure on $\ii$)  and for every $v$ in $\ii$ $$\sigma(v):E\Rightarrow \  \pitchfork(\ddot{\mu}_{r},\alpha)(v,{\unit}^{\op})(E) $$ is a natural isomorphism. Here, we note that $\pitchfork(\ddot{\mu}_{r},\alpha)({\unit} ,u^{\op})(E) =E$ as $\alpha$ is mute on the right. Write $\alpha^v=\alpha(v,\unit^{\op})$ and $\sigma(v)_u=\sigma_{u,v}$ (while noting that $\alpha(v,\unit^{\op})=\alpha(v,w^{\op})$ for every $w$ in $\ii$). Then, for every $u,v$ in $\ii$, $\sigma$ defines an isomorphism  $\sigma_{u,v}:E_u\to \alpha^vE_{u{\circledast} v}$. The diagram \ref{diag:triangleofend} implies that for every $w$ in $\ii$ the following triangle commutes
	\begin{equation}\label{diag:triangleofstab}\begin{tikzpicture}[scale=1.55]
		\node (AB2) at (2,1) {$ E_{u}$};
		\node (AB) at (1,0) {$\alpha^vE_{u {\circledast} v}$}; 
		\node (BC) at (4,0) { $\alpha^{v{\circledast} w}E_{u {\circledast}  v {\circledast}  w}$};
		\path[->,font=\scriptsize,>=angle 90]
		(AB) edge node[below]{ }(BC)
		(AB2) edge node[left] {$\sigma_{u,{v}}$} (AB)
		(AB2) edge node[right] {$\sigma_{u,{v}{\circledast} {w}}$} (BC);
\end{tikzpicture}\end{equation}
where the bottom horizontal map is the composition 
$$\alpha^vE_{u {\circledast} v}\stackrel{\alpha^{v}(\sigma_{u {\circledast} v,w})}\lto\alpha^{v}\alpha^{w}E_{u {\circledast}  v {\circledast}  w} \stackrel{\cong}\lto \alpha^{v{\circledast} w}E_{u {\circledast}  v {\circledast}  w}$$
Morphisms are natural transformations that are coherent with the $\ii$-action, and weak equivalences are defined levelwise. This category admits a right action $\tilde{\alpha}$ given by $\tilde{\alpha}({u},{v}^{\op})(f,\sigma)=(f\circ \ddot\mu_r({u}),\tilde{\sigma}({u}))$ where $\tilde{\sigma}({u})_{{v},{w}}=\sigma{({u}{\circledast} {v})_{w}}$. 

On the other hand, the lax stabilization of $\cA$ with respect to $\alpha$ is the same except that $\sigma$ in the above construction is not required to be an isomorphism but just a natural transformation.	In the case when the underlying relative category $\cA$ is a model category and $\ii$ acts by left Quillen functors, the lax stabilization coincides with the (dual version) lax homotopy limit in \cite{bergner2012homotopy}; and therefore, admits a model structure where weak equivalences and fibrations defined levelwise.

{ One equivalently uses $\Inv(\cA_{\alpha})$ for the stabilization, which has a natural left action. In this case, objects can be viewed as pairs $(E,\varsigma)$ where $E:\ii\times \ii\to A$ and is a functor and $\varsigma:E\to \pitchfork(\ddot{\mu}_{r}\times \ddot{\mu}_{l},\alpha)(E)$ is a natural isomorphism; so that, for every $u,v,w,z$ in $\ii$, $\varsigma_{u,v;w,z}:E(u,v)\to \alpha(w,z^{\op})E(u{\circledast} w,z{\circledast} v)$ is a natural isomorphism. Writing $E^{u\ominus v}=E(u,v)$ and $\alpha^{w\ominus z} =\alpha(w,z^{\op})$, we can see that  $\Inv(\cA_{\alpha})$ just extends the indexing in $\Inv_l(\cA_{\alpha})$ to formal differences in $\ii$. Again, in the lax version, the assumption on $\varsigma$ being invertible is dropped. }

\subsubsection{A construction of spectrification}\label{sssec:constructionofspectrification}
For the case of the actions on relative categories, we can construct spectrification as in \cite{elmendorfkriz} or \cite{lewis1986equivariant}.  Let $\ii$ be symmetric monoidal and $\alpha$ be an $\ii$-action on a relative category $\cA$. Assume that $\alpha$ is mute on the right. Recall that spectrification is left adjoint to the natural inclusion $\iota_\alpha: \Inv_{s}(\cA_{\alpha})\to \ell \Inv_{s}(\cA_{\alpha})$. Let $\Tr(\ii)$ denote the transport category of $\ii$; i.e., category whose objects are objects of $\ii$ and a morphism $u\to v$ in $\Tr(\ii)$ is a pair $(w,f)$ where $w$ is an object and $f:u{\circledast} w\to v$ is an isomorphism in $\ii$. The composition of $(w,f)$ and $(z,g)$ is $(z{\circledast} w, g\circ f{\circledast} \id_w)$. Suppose that $\alpha$ is an action that commutes with  $\Tr(\ii)$ shaped colimits. Then  the inclusion $\iota_\alpha$ admits a left adjoint $L$ given as follows: For $(E,\sigma)$ an object in $\ell \Inv_{s}(\cA_{\alpha})$, let $F^{\alpha,E,z}:\Tr(\ii)\to \cA$ be the functor given by $F^{\alpha,E,z}(u)=\alpha(u,\unit^{\op})E({z{\circledast} u})$ and $F^{\alpha,E,z}((w,f))$ is the composition 
$$\alpha(u,\unit^{\op})E({z{\circledast} u}) \overset{u\cdot \sigma}\lto  \alpha(u{\circledast} w,\unit^{\op})E({z{\circledast} u{\circledast} w}) \overset{f^*}\cong  \alpha(v,\unit^{\op})E({z{\circledast} v})$$
Then  
$L E:\ii\to \cA$ is the functor given by $$L E(z)=\colim_{\Tr(\ii)} F^{\alpha,E,z}$$
We call the left adjoint of this inclusion \emph{spectrification}. The $x$ component of $\sigma$ is given by the following composition 
$$\alpha(x,\unit^{\op})(L E(z{\circledast} x))=\alpha(x,\unit^{\op})(\colim_{u\in \Tr(\ii)}\alpha(u,\unit^{\op})E(z{\circledast} x{\circledast}  u))$$ $$
= \colim_{u\in \Tr(\ii)}\alpha(x{\circledast} u,\unit^{\op})E(z{\circledast} x{\circledast}  u)\cong \colim_{u\in \Tr(\ii)}\alpha( u,\unit^{\op})E(z{\circledast}  u)=L E(z)
$$
so that $L E$ is an object in $\ii\frM$.  The naturality of $\sigma$ follows from the naturality of colimit. The morphism $L$ is $\ii$-equivariant via  $L E(x{\circledast} -)\cong L(E(x{\circledast} -))$.


\subsection{Homotopy stabilizations of $\ii$ actions on a model category}\label{ssec:specialcase}
Consider the special case when the symmetric monoidal category $\ii$ is a groupoid; that is, the $2$-category $\mathbf{B}\ii$ is a $(2,1)$-category and $\alpha$ is an action on a relative category $\cA$ given by a $2$-functor from $\HH$ to the $2$-truncation of the $(\infty,1)$-localization of the model category $\RelCat$ with respect to the Barwick-Kan model structure. Then $\Inv_l(\cA_\alpha)$, as given by an weak end, coincides with the homotopy end (i.e., $(2,1)$-limit). In the case when the relative category $\cA$ is a model category and $\ii$ acts on $\cA$ by left Quillen functors, then following \cite{bergner2012homotopy} one obtains a construction of $\Inv_l(\cA_\alpha)$ in this case same as above. Objects of $\Inv_l(\cA_\alpha)$ are pairs $(E,\sigma)$ where $E:\ii\to A:u\to E_u$ is a functor and $$\sigma(v):E\Rightarrow \  \pitchfork(\ddot{\mu}_{r},\alpha)(v,{\unit}^{\op})(E) $$ is a natural weak equivalence for every $v$ in $\ii$; i.e., $\sigma_{u,{v}}$ is a weak equivalence. The diagram  \ref{diag:triangleofstab} commutes, which corresponds to the compatibility condition mentioned in \cite[Defn. 3.1]{bergner2012homotopy}.  Observe that $\Inv_l(\cA_\alpha)$ here is homotopically correct as its $(\infty,1)$-localization  is stable with respect to the induced action.

Even though the underlying relative category $\cA$ is a model category, its stabilization $\Inv_l(\cA_\alpha)$ does not need to have a model structure. On the other hand, its lax stabilization $\ell \Inv_l(\cA_\alpha)$ admits a model structure; namely, the level model structure, see also \cite{bergner2012homotopy}. If there exist a spectrification functor $L: \ell\Inv_l(\cA_\alpha)\to \Inv_l(\cA_\alpha)$ and the left Bousfield localization of the model structure on the lax stabilization with respect to maps inverted by $L$ (from the weak equivalences in the stabilization) exists, then the new model structure defines a stabilization, as its homotopy category is equivalent to the homotopy category of  $\Inv_l(\cA_\alpha)$ (see also \cite{bousfieldfriedlander}). Here, note that the spectrification together with the inclusion of stabilization into lax stabilization defines a deformation retract of homotopical categories in the sense of \cite{dhks}, therefore, have equivalent homotopy categories.

On the other hand, the costabilization in this case coincides with a Grothendieck-type construction, see \cite{bergner2014hocolim} and \cite{harpaz2015grothendieck}, which is a homotopy colimit in $\RelCat$ with respect to the Barwick-Kan model structure. As in the case of the stabilization, $(\infty,1)$-localization  of costabilization obtained in this way is stable with respect to the induced action.

\subsection{Some known examples of stabilizations of relative categories}\label{ssec:examplesofstab}
In each of the following examples, the monoidal category that acts is symmetric monoidal and actions are mute on the right. Thus,  $\Stab_{-}$ can be identified by $\Inv_{r}$ and $\ell\Stab_{-}$ can be identified by $\ell\Inv_{r}$. Unless otherwise states, we use the special case given in  \ref{ssec:specialcase}. Therefore, $\ii$ is a monoidal groupoid and $\RelCat$ denotes the $(2,1)$-category of relative categories. Thus, the weak end will coincide with the homotopy end. We use the constructions of $(2,1)$-limits as described in \cite{bergner2012homotopy}.

\subsubsection{Sequential Spectra}\label{sssec:sequentialspectra} Let $\ii=\bbN$ with the monoidal structure given by addition and with only identity  morphisms, and let $\cA=\cT$ be the category of pointed spaces  with the standard model structure. Let $\Omega:\HH\to \RelCat$ be the action given by usual loop space functors; that is, $\Omega(n,m)(X)=\Omega^nX=\Omega\cdots \Omega X$, the $n$-fold loop space of $X$. This action is in fact the action generated by the functor $\Omega:\cT\to\cT$. Then  $\Inv_{r}(\cT_\Omega)$ has objects as pairs $(E,\sigma)$ where $E:\bbN\to \cT:n\mapsto E_n$ is a functor and   $\sigma_m:E_n\to \Omega^mE_{n+m}$ is a weak equivalence natural in $n$ and $m$. The triangular diagram \ref{diag:triangleofstab} above already commutes in this case. Such objects are known as spectra in the literature. If we consider the lax stabilization $\ell\Inv_{r}(\cT_\Omega)$, then $\sigma_m$ is just a map. Objects of  $\ell\Inv_{r}(\cT_\Omega)$ are then called  prespectra. The model structure on  $\ell\Inv_{r}(\cT_\Omega)$  coincides with the level model structure and the spectrification functor in \ref{sssec:constructionofspectrification} is the usual spectrification of prespectra. 
\subsubsection{Spectra indexed by standard spheres}\label{ssec:generalspheres} 	
Note that our original motivation of relating spectra with actions of monoidal categories follows from the fact that loop space functors generate a monoidal category with natural transformations between them, which is $\cT$-enriched. Let $\overline{\ii}$ be this monoidal category. Since $\Omega^n\cong \cT(S^n,-)$, by Yoneda lemma we can identify $\overline{\ii}$ by the monoidal category with objects pointed spheres $S^n$ for each $n$ in $\bbN$ and with morphisms $\ii(S^n,S^m)\cong \cT(S^n,S^m),$ set of based maps between pointed spheres. This category is symmetric monoidal with the usual smash product. Let $\ii$ be the core of this monoidal category. Then   $\Inv_{r}(\cT_\Omega)$  is  the category whose objects as pairs $(E,\sigma)$ where $E:\ii\to G\cT:S^n\mapsto E_n$ is a functor and $\sigma_{n,m}:E_n\to \Omega^mE_{n+ m}$ is an weak equivalence natural in both $n$ and $m$ in the sense that for every map $f:S^n\to S^n$ the following diagrams commute  
 \[\begin{tikzpicture}[scale=.8]
	\node (0) at (1,0) {$E_n$};
	\node (c) at (6,0) {$\Omega^mE_{n+m}$};
	\node (a) at (1,2) {$E_n$};
	\node (b) at (6,2) {$\Omega^mE_{n+m}$};
	\path[->,font=\scriptsize,>=angle 90]
	(a) edge node[above]  {$\sigma_{n,m}$} (b)
	(b) edge node[right] {$\Omega^mE(f\wedge S^m)$} (c)
	(a) edge node[left] {$E_f$} (0)
	(0) edge node[above] {$\sigma_{n,m}$ } (c);
\end{tikzpicture} \quad \quad
\begin{tikzpicture}[scale=.8]
	\node (0) at (1,0) {$\Omega^nE_{t+n}$};
	\node (c) at (6,0) {$\Omega^nE_{t+n}$};
	\node (a) at (1,2) {$E_t$};
	\node (b) at (6,2) {$\Omega^kE_{t+n}$};
	\path[->,font=\scriptsize,>=angle 90]
	(a) edge node[above]  {$\sigma_{t,n}$} (b)
	(b) edge node[right] {$\Omega^fE(t+ n)$} (c)
	(a) edge node[left] {$\sigma_{t,n}$} (0)
	(0) edge node[above] {$\Omega^n(E(S^t\wedge f))$ } (c);
\end{tikzpicture}\] 

\subsubsection{Coordinate Free  spectra}\label{ssec:coordinatefreespectra} Let $\ii$ be the core of the monoidal category of real inner product spaces with direct sum and linear isometric embeddings. For an inner product space $V$, let $S^V$ denote its one point compactification. Define a $\ii$-action $\Omega$ on the relative category $\cT$ by  $\Omega(V,\unit^{\op})(X)=\Omega^VX=\cT(\bbS^V,X)$. Then  $\Inv_{r}(\cT_\Omega)$ is equivalent to the category with objects as pairs $(E,\sigma)$ where $E:\ii\to \cT:V\mapsto E_V$ is a functor and  $\sigma_{V,W}:E_V\to \Omega^WE_{V\oplus W}$ is a natural weak equivalence. Moreover, for any $Z\in \ii$ the diagram
	\[\begin{tikzpicture}[scale=1.55]
		\node (AB2) at (1,1) {$E_{V}$};		
		\node (BC2) at (4,1) {$\Omega^{W\oplus Z}E_{V\oplus W\oplus Z}$};
		\node (AB) at (1,0) {$\Omega^WE_{V\oplus W}$};		
		\node (BC) at (4,0) {$\Omega^{W} \Omega^{Z} E_{V\oplus W\oplus Z}$};

		\path[->,font=\scriptsize,>=angle 90]
		(AB) edge node[below]{$\Omega^W(\sigma_{V\oplus W,Z})$} (BC)
		(AB2) edge node[above]{$\sigma_{V,W\oplus Z}$} (BC2)
		(AB2) edge node[left]{$\sigma_{V,W}$} (AB)
		(BC2) edge node[right]{$\cong$} (BC);
	\end{tikzpicture}\]
commutes. Note that this diagram is  the same as the  diagram \ref{diag:triangleofstab}  after composing bottom arrow with  inverse of the right vertical isomorphism.
A refinement of the construction above gives the coordinate free spectra indexed by a universe. Let $U=\mathbb R^{\infty}$ be a  countably infinite dimensional real inner product space. Let $\ii$ be the category whose objects are finite dimensional subspaces of $U$ and morphisms are linear isometric isomorphisms.   Again, define $\Omega$ on $\cT$ by  $\Omega(V,1)(X)=\Omega^VX=\cT(\bbS^V,X)$. Then  the stabilization $\Inv_{r}(\cT_\Omega)$ is equivalent to the category of coordinate free spectra. In fact, if $V\subset W$ by writing $W - V$ for the orthogonal complement of $V$ in $W$, we obtain structure maps of the form   $\sigma_{V,W- V}:E_V\to \Omega^{W- V}E_{ W}$.  Similarly, the lax stabilization $\ell\Inv_{r}(\cT_\Omega)$ gives the coordinate free prespectra, which admits the level model structure in which weak equivalences and cofibrations defined levelwise.  Since $\ii$ is a monoidal groupoid, the construction in \ref{ssec:specialcase} gives the category of coordinate free $\Omega$-spectra.

	\subsubsection{Genuine $G$-Spectra}\label{ssect:genuinespectra}  	Now, let $G$ be a group and $\ii$ be the monoidal category of $G$-inner product spaces with direct sum and with $G$-equivariant linear isometric isomorphisms. Consider the category $G\cT$ of pointed $G$-spaces and pointed $G$-maps with the standard model structure.  Define an $\ii$-action $\Omega:\HH\to \RelCat$ on $G\cT$  by  $\Omega(V,\unit^{\op})(X)=\Omega^VX=\overline{G\cT}_*(\bbS^V,X)$, the set of pointed maps from the one point compactification of $V$ to $X$ considered with the conjugation action. Then  $\Inv_{r}(G\cT_\Omega)$ is  the category with objects as pairs $(E,\sigma)$ where $E:\ii\to G\cT:V\mapsto E_V$ is a functor and $\sigma_{V,W}:E_V\to \Omega^WE_{V\oplus W}$ is a weak equivalence that is natural in both $V$ and $W$. The diagram analogues to \ref{diag:triangleofstab} commutes. 
	
For convenience one considers the more classical  notion of equivariant spectra  indexed by a $G$-universe.  Let $\ii$ be as in \ref{sssec:ROGradedCohomologyTheories}.  Define an $\ii$-action $\Omega$  on $G\cT$ as above. Then $\Inv_{r}(G\cT_\Omega)$ is the category of genuine $G$-spectra and  $\ell\Inv_{r}(G\cT_\Omega)$ gives the category of  genuine orthogonal $G$-prespectra, see \cite[Ch. XII, 2]{may}. Here, we note that $\ii(V,V)=O(V)$, so that for each object $(E,\sigma)$ in $\ell\Inv_{r}(G\cT_\Omega)$, $E(V)$ admits a $O(V)$-action and $\sigma_{{V} ,{W}}$ is $O(V)\times O(W)$-equivariant analogous to  the coordinate free case above.
	
	\subsubsection{Parameterized Spectra}\label{ssect:paraspectra} Let $B$ be a space and $\cT_B$ be the model category given in \cite[Ch. 6]{mayparametrized}. Let $\ii$ be the monoidal category as given in \ref{sssec:bundlegradedparameterizedtheories}.  
	  Consider the model structure on $\cT_B$  \cite[Ch. 6]{mayparametrized}. For $\xi $ in $\ii$, denote by $\bbS^\xi$ the associated fiber-wise one point compactification, which is a sphere bundle over $B$ with point as the section induced by the zero section of $\xi$. Define the $\ii$-action $\Omega_B:\HH\to \RelCat$ on  $\cT_B$ as follows: Given $\xi,\eta$ in $\ii$, define $\Omega_B$ by \[\Omega_B(\xi,\eta^{op})(t:X\to B)=\cT_B(\bbS^{\xi},  t).\] Denote by $\Omega_B^\xi$ the functor $\cT_B(\bbS^{\xi},-)$. Then  $\Inv_{r}((\cT_B)_{\Omega_B})$ has objects as pairs $(\digamma,\sigma)$ where $\digamma:\ii\to \cT_B:\xi\mapsto \digamma_\xi$ is a functor and $\sigma_{\xi,\eta}:\digamma_\xi\to \Omega_B^\eta \digamma_{\xi\oplus_B \eta}$.  The diagram analogues to \ref{diag:triangleofstab} commutes. 	Such objects are called as genuine spectra parameterized by $B$.  The lax stabilization  $\ell\Inv_{r}((\cT_B)_{\Omega_B})$ then  gives the category of genuine prespectra parameterized by $B$. The category $\ell\Inv_{r}((\cT_B)_{\Omega_B})$ admits the level model structure as in the previous examples.

	\subsubsection{Parameterized $G$-Spectra} Parameterized $G$-spectra and prespectra can be obtained just like the transition from coordinate free spectra to genuine $G$-spectra.

\subsubsection{Diagram Spectra}\label{ssec:diagramspectra} Let $\VV$ be a closed symmetric monoidal model category with internal hom $\hom_{\VV}$ and $\frM=\VV\Cat$, the $\VV\Cat$-category of $\VV$-cat enriched categories. Let $\ii$ be a $\VV$-enriched small symmetric monoidal category. Assume both $\VV$ and $\ii$ are pointed as categories. Let $[\ii,\VV]$ be the $\VV$-category of $\VV$-enriched functors from $\ii$ to $\VV$. The category  $[\ii,\VV]$ admits a symmetric monoidal product with Day convolution.  Let $R$ be any monoid in the symmetric monoidal category $[\ii,\VV]$. Define an $\ii$ action on $\Omega:\HH\to \VCat$ on $\VV$ as $$\Omega(u,v^{\op})(X)=\Omega^uX=\hom_{\VV}(R(u),X).$$ Since $R$ is a monoid, the action is well-defined. The lax stabilization with respect to $\Omega$, $\ell\Inv_{r}(\VV_\Omega)$, has objects as pairs $(E,\varsigma)$ where $E:\ii\to \VV:u\to E_u$ and $\varsigma_v:E_v\to \Omega^u E_{v{\circledast} u}$ is a natural transformation. If $\VV$ is a relative category then so does  $\ell\Inv_{r}(\VV_\Omega)$ with the levelwise weak equivalences.  If $\Omega$ factors through the category of adjunctions; i.e., there exist an adjoint action $\Sigma$ on $\VV$,  then a $\ii$-spectra as given in \cite{MMSS} can be seen as an object in the lax stabilization provided that $\VV=\cT$ and $\Omega$ is the ordinary loop space action. In this case, $\ell\Inv_{r}(\VV_\Omega)$ admits a model structure, see  \cite{MMSS}. 

In particular, we can obtain symmetric spectra, orthogonal spectra and their equivariant versions by choosing $\ii$ accordingly. If we choose $\ii$ as the permutation category;  whose objects are finite sets $[n]=\{1,\dots,n\}$ and whose morphisms are permutations (i.e., $\ii$ is the skeleton of the groupoid of finite sets and permutations), and cardinal sum as the monoidal product and $[0]=\emptyset$ as the monoidal unit, and if we choose the usual loop space action on $\VV=\Top$, then the resulting stabilization gives the category of symmetric spectra \cite{shipleysymmetric}, \cite{schwedespectra}. If we choose $\ii$ to be the topological groupoid of Euclidean spaces and linear isomorphisms, i.e., objects of $\ii$ are $\bbR^n$ for all $n\in \bbN$ and morphisms are linear isometric isomorphisms, than the resulting stabilization gives the orthogonal spectra \cite{mandellmay}. Both  examples are given in \cite{MMSS} as well. One can also take the action of the topological monoidal category $\ii$ in Section \ref{ssec:generalspheres} of all pointed spheres and pointed homeomorphisms, and obtain another model for the category of spectra.
	\subsection{Costabilization of relative categories with respect to $\ii$-actions}
In each of the following examples we assume the acting monoidal category is symmetric monoidal and actions are mute on the left, so that  $\coStab_{-}$ can be identified by $\coInv_{l}$ and $\ell\coStab_{-}$ can be identified by $\ell\coInv_{l}$. Again, we use the dual version of the special case given in  \ref{ssec:specialcase}; hence, $\ii$ is a monoidal groupoid and $\RelCat$ is the $(2,1)$-category of relative categories as described there. Thus, the weak coend will coincide with the homotopy coend and we can use the constructions  described in \cite{bergner2014hocolim} and \cite{meier2014homotopy}.

	Let $\alpha:\HH\to \RelCat$ be an $\ii$-action on $\cA=(A,W_A)$ that is mute on the {left}.   The  costabilization of $\cA$ with respect to $\alpha$, denoted by $\coStab_{\ii}(\cA)$, can be defined as the weak coend of $\alpha$, which can be seen as a homotopy colimit over $\ii$. In the case when $\ii$ is symmetric monoidal, this homotopy colimit can be given in terms of Gröthendieck-type construction. Since $\alpha$ is mute on the left, we have $\coStab_{\ii}(\cA)=\int_{\ii}(\ddot{\mu}_{r}\odot \alpha)\circ \iota_r$. Then objects of $\coStab_{\ii}(\cA)$ are pairs $({u},a)$ with $({u}\in \ii$ and $a\in \cM$, and a morphism between $({u},a)\to ({u}',a')$ is a triple $({v},\varphi,f)$ where ${v} \in \ii$ and $(\varphi,f):({u}{\circledast} {v}, \alpha({\unit},{v}^{\op})(a))\to ({u}',a')$. 	A triple $({v},\varphi,f)$ is a weak equivalence if $(\varphi,f):({u}{\circledast} {v}, \alpha({\unit},{v}^{\op})(a))\to ({u}',a')$ is a weak equivalence; that is, ${u}{\circledast} {v}\cong {u}'$ in $\ii$ and $f:\alpha({\unit},{v}^{\op})(a)\to a'$ is a weak equivalence in $\cM$. In particular, for every ${v}$ in $\ii$, $({u},a)$ is weakly equivalent to $({u}{\circledast} {v},\alpha({\unit},{v}^{\op})(a))$.
	
In sufficiently nice cases, if $\cA$ is a model category $\coStab_{\ii}(\cA)$ admits a model structure; e.g., when $\ii$ is bicomplete (so a model category with trivial model structure), and $\alpha$ is relative and proper functor, see \cite{harpaz2015grothendieck}. Observe that as it coincides with the homotopy colimit; and thus, this construction is homotopically correct with respect to the Barwick-Kan model structure (see \cite{bergner2014hocolim}, \cite{harpaz2015grothendieck} and \cite{meier2014homotopy}).

	\subsection{Some examples of costabilizations  of relative categories}\label{ssec:laxcostabofrelativecats}
	
	\subsubsection{Spanier Whitehead Category} Let $\ii=\bbN$ with only identity maps as morphisms and $\cA=\cT$ pointed topological spaces with with standard model structure and $\Sigma:\HH\to \RelCat$ be the action given by usual suspensions; that is, $\Sigma(n,m)(X)=\Sigma^nX$ the $n$-fold suspension of $X$. Then  $\coStab_{\ii}(\cT)$ has objects as pairs $(n,X)$ and for $n\leq m$ a morphism from $({n},X)$ to $ (m,Y)$ is a map $f:\Sigma^{m-n} X\to Y$. 	Its homotopy category with respect to the weak equivalences described above gives the usual Spanier-Whitehead category (see also \cite{meier2014homotopy}).

	\subsubsection{Coordinate Free Spanier-Whitehead Category} Let $\ii$ be as in the case of the coordinate free spectra  and $\Sigma:\HH\to \RelCat$ be the action given by  $\Sigma(V,W)(X)=\Sigma^WX:= S^W\wedge X $ where $\wedge $ denotes the smash product of spaces. Then  $\coStab_{\ii}(\cT)$ has objects as pairs $(W,X)$ and for $\dim(W)\leq  \dim(Z)$ a morphism from $({W},X)$ to $ (Z,Y)$ is a triple $(V,\phi,f)$ where $\phi:V\oplus W\to Z $ is an  isometric isomorphism and  $f:\Sigma^{W} X\to Y$ is a map. Such a triple $(V,\phi,f)$ is a weak equivalence if $f$ is a weak equivalence. Then, following the homotopy category $\coStab_{\ii}(\cT)$ with these weak equivalences gives the coordinate free version of the Spanier-Whitehead category.

	\subsubsection{$G$-Equivariant Spanier-Whitehead Category} This is the $G$-equivariant version of coordinate free Spanier-Whitehead Category where $\ii$ is chosen as in \ref{ssect:genuinespectra} and $\Sigma:\HH\to \RelCat$ be the action on $G\cT$ given by  $\Sigma(V,W)(X)=\Sigma^WX:=\cT(S^W,X)$ where $X$ is a $G$ space and $S^W$ is the one point compactification of $W$ with the induced $G$-action. Then   $\coStab_{\ii}(\cT)$ has objects as pairs $(W,X)$ and for $\dim(W)\leq  \dim(Z)$ a morphism from $({W},X)$ to $ (Z,Y)$ is a triple $(V,\phi,f)$ where $\phi:V\oplus W\to Z $ is an equivalence class of a isometric $G$-isomorphisms and  $f:\Sigma^{W} X\to Y$ is a $G$-map. 
	\subsubsection{Parameterized Spanier-Whitehead Category} This is the parameterized  version of coordinate free Spanier-Whitehead category where $\ii$ is chosen as in \ref{ssect:paraspectra}. Let $\Sigma_B:\HH\to \RelCat$ be the $\ii$-action on $\cT_B$ given by  $\Sigma_B(\zeta,\xi)(t)=\Sigma^\xi t:=S^\xi\wedge_B t$ where $t:X\to B$ be a pointed map over $B$ (i.e., a pointed object in the over category $\cT/B$) and $S^\xi$ is the sphere bundle over $B$ given by the fiber-wise one point compactification of $\xi$, and $\wedge_B$ is the fiber-wise smash product. 	In this case objects are pairs $(\xi,t)$ where $\xi$ is an object in $\ii$ and $t:X\to B$ a pointed map. A morphism from $(\xi,t)$ to $(\zeta,z)$ is a triple $(\chi,m,f)$ where $\chi$ is an object in $\ii$ and $m:\xi\oplus \chi \to \zeta $ a morphism in $\ii$ and $f:S^\chi\wedge_B t \to z$  is a pointed map over $B$.
	\subsubsection{Equivariant Parameterized Spanier-Whitehead Category} The description of Equivariant Parameterized Spanier-Whitehead Category just the equivariant generalization of the Parameterized Spanier-Whitehead Category above, which is the costabilization of $G\cT_B=\id_B/(G\cT/B)$ with respect to the action given by the fibrewise suspensions as above, but equipped with $G$-actions.

	\subsubsection{Diagram Spanier-Whitehead Category}
Let $\VV$, $\ii$ and $R$ be as in \ref{ssec:diagramspectra}. Define a $\ii$ action on $\Sigma$ on $\VV$ as $$\Sigma(u,v^{\op})(X):=\Sigma^v X:= X \wedge_{\VV}R(u).$$
 Then   $\coStab_{\ii}(\VV)$ has objects as pairs $(v,X)$ where $v$ in $ \ii$ and $X$ in $\VV$. A morphism from $(w,X)$ to $ (v,Y)$ is a triple $(u,\phi,f)$ where $\phi:u\wedge_\ii w\to v$ is a morphisms in $\ii$ and  $f:X \wedge_{\VV}R(u)\to Y$ is a morphism in $\VV$. Weak equivalences are the triples in which $\phi$ is an isomorphism and $f$ is a weak equivalence.
 
\begin{remark}
	Observe that the duality between   stabilization and   costabilization gives that the constructions of stable homotopy category as the homotopy category of versions of spectra and as versions of Spanier-Whitehead category are dual. 
\end{remark}

 	\section{Some concluding remarks}\label{sec:concluding}
We conclude our paper with some remarks regarding stabilizations in the $2$-category of (symmetric) monoidal relative categories. These aspects are stated as future directions and considerations of the main theme of this paper. Here we give a brief outline about how the constructions can be applied to monoidal stabilizations. 
 	\subsection{On stabilizations of (symmetric) monoidal relative categories}\label{ssec:monoidalstructuresrelcat}
 Let $\MonCat$ and $\Sym\MonCat$ denote the categories of monoidal and symmetric monoidal categories. It is known that both  $\MonCat$ and $\Sym\MonCat$ are categories of models on $2$-Lawvere theories (also known as doctrines). We refer to  \cite{yanofsky2000lawvere2theory},  \cite{yanofskyhomotopy} and \cite{shulman2019practical} for more details about $2$-theories, and their examples. We can use an internal approach to define (symmetric) monoidal relative categories. We denote by $\TT_{\Mon}$ and $\TT_{\Sym}$ the $2$-Lawvere theories (i.e., syntactic categories) for $\MonCat$ and $\Sym\MonCat$, respectively. The theories $\TT_{\Mon}$ (resp. $\TT_{\Sym}$) can be taken as the opposite categories of free monoidal (resp. symmetric monoidal) categories on $n$-generators, see e.g. \cite[Ex. 2.3]{yanofsky2000lawvere2theory}.

  Then we can define a monoidal relative category as a product preserving $2$-functor $A:\TT_{\Mon}\to \RelCat$ and denote the $2$-category of monoidal relative categories by $\Mon\RelCat$. Similarly, a symmetric monoidal relative category as a product preserving $2$-functor $A:\TT_{\Sym}\to \RelCat$ and denote the $2$-category of monoidal relative categories by $\Sym\RelCat$.  	Recall that we assume every relative category is semi-saturated. Therefore, the homotopy category functor $\ho:\RelCat\to \Cat$ preserves finite products, see   \cite{barwick2012relative}. Thus, the homotopy category of a (symmetric) monoidal relative category defines a (symmetric) monoidal category. If we consider in the strict sense, being models over a $2$-Lawvere theories both  $\Mon\RelCat$ and $\Sym\RelCat$ have limits, which are created in $\RelCat$. Thus, if $\alpha$ is an $\ii$-action on a (symmetric) monoidal relative category $\cA$, then $\Stab_{\ii}(\cA_\alpha)$ is a (symmetric) monoidal category that is stable with respect to $\ii$-action. However, it is often hard to find examples in the strict sense. On the other hand, if we consider these categories with strong monoidal functors then examples of such actions are plentiful (e.g., classical loop space or suspension actions). In this case, these categories have weak limits, however, does not have to be created in $\RelCat$. Thus, for any $\ii$-action $\alpha$ on a (symmetric) monoidal category $\cA$ where $\ii$ acts by strong monoidal functors, the stabilization $\Stab_{\ii}(\cA_\alpha)$ is a  (symmetric) monoidal category that is stable with respect to the action. Since stabilizations is defined up to equivalence, there is a stabilization in $\RelCat$ after forgetting the monoidal structure which is equivalent to the underlying relative category.
 
 In the case $\ii$ is a symmetric monoidal groupoid (i.e., $B\ii$ is a $(2,1)$-category), we can define homotopy stabilizations as $(2,1)$-limits (i.e., homotopy limits \cite[Sec. 4.2.4]{lurie}) in the $(2,1)$-category of obtained from the homotopy category of $\Sym\RelCat$, where the homotopical structure on $\Sym\RelCat$ is induced from the Barwick-Kan model structure on $\RelCat$ via the forgetful functor $\Sym\RelCat\to\RelCat$. We will refer to the stabilization here as homotopy stabilization, as it is defined via a homotopy limit. Note that it satisfies the universal property up to homotopy in $\Sym\RelCat$.

 More generally, one considers stabilizations in arbitrary  $2$-categorical algebraic theories and considers stabilizations with extra algebraic structure.

\subsection{On (co)stabilizations of $(\infty,1)$-categories with respect to actions}
By using the model independent considerations of  $(\infty,1)$-categories of \cite{riehl2018elements} and setting $\frM$ to be the  homotopy $2$-category  of $(\infty,1)$-categories, $(\infty,1)$-functors and $(\infty,1)$-natural transformations for a suitable $\infty$-cosmos so that all weak $2$-limits and $2$-colimits exists in $\frM$, along with weak $\Cat$-powers and weak $\Cat$-copowers. Let $\alpha$ be an $\ii$-action on an $(\infty,1)$-category $\cA$. Then, $\Stab_{\ii }(\cA_{{\alpha}})$ defines the stabilization and $\coStab_{\ii }(\cA_{{\alpha}})$ defines the costabilization of $\cA_\alpha$. Note that $\Stab_{\bbN}(\cA_{{\Omega}})$ is equivalent to the usual definition of stabilization of $(\infty,1)$-categories as in \cite[Sec. 7]{lurie2012higher}. 

We can also obtain (symmetric) monoidal stabilizations by replacing the $2$-Lawvere theories mentioned in Section \ref{ssec:monoidalstructuresrelcat} by their $\infty$-category versions, and repeating the procedure mentioned there and stabilize  (symmetric) monoidal $\infty$-categories with respect to given actions as a weak limit (i.e., a weak end) within the category of  (symmetric) monoidal $\infty$-categories, and obtain (symmetric) monoidal stable $\infty$-categories that are stable with respect to an action.
\bibliographystyle{abbrv}
\bibliography{research}

\end{document}